\documentclass[a4paper,twoside,10pt]{article}

\usepackage[a4paper,left=3cm,right=3cm, top=3cm, bottom=3cm]{geometry}
\usepackage[latin1]{inputenc}
\usepackage{mathrsfs}
\usepackage{graphicx}
\usepackage{epstopdf}
\usepackage{xcolor}
\usepackage{amsmath}
\usepackage{amsthm}
\usepackage{amssymb}
\usepackage{stmaryrd}
\usepackage{float}
\usepackage{bigints}
\usepackage{cite}
\usepackage{color}
\usepackage[abs]{overpic}
\usepackage[font=footnotesize,labelfont=bf]{caption}
\usepackage{cases}
\usepackage{tikz}
\usepackage{algorithm,algpseudocode}
\usepackage{rotating}
\usepackage{blkarray}
\usetikzlibrary{matrix,calc,arrows}
\usepackage[author={Lorenzo}]{pdfcomment}
\usepackage{soul,xcolor}
\usepackage{verbatim}
\usepackage{graphicx}
\usepackage{subcaption}
\usepackage{algorithm}
\usepackage{pifont}
\usepackage{multirow}

\restylefloat{table}
\theoremstyle{plain}
\newtheorem{theorem}{Theorem}[section]
\newtheorem{corollary}[theorem]{Corollary}
\newtheorem{lemma}[theorem]{Lemma}

\theoremstyle{definition}

\theoremstyle{remark}
\newtheorem{remark}{Remark}

\newcommand{\eremk}{\hbox{}\hfill\rule{0.8ex}{0.8ex}}
\newcommand{\Norm}[1]{{\left\|{#1} \right\|}}
\newcommand{\SemiNorm}[1]{{\left|{#1} \right|}}

\newcommand{\p}{p}
\newcommand{\h}{h}
\newcommand{\hbf}{\mathbf h}
\newcommand{\xbf}{\mathbf x}
\newcommand{\E}{K}
\newcommand{\Ep}{\E^+}
\newcommand{\Em}{\E^-}
\newcommand{\vp}{v^+}
\newcommand{\vm}{v^-}
\newcommand{\hE}{\h_\E}
\newcommand{\F}{F}
\newcommand{\hF}{\h_\F}
\newcommand{\e}{e}

\newcommand{\Ccal}{\mathcal C}
\newcommand{\taun}{\mathcal T_n}
\newcommand{\Fcaln}{\mathcal F_n}
\newcommand{\FcalnI}{\Fcaln^I}
\newcommand{\FcalnB}{\Fcaln^B}
\newcommand{\Vcaln}{\mathcal V_n}
\newcommand{\Nbb}{\mathbb N}
\newcommand{\Pbb}{\mathbb P}
\newcommand{\Qbb}{\mathbb Q}
\newcommand{\Rbb}{\mathbb R}
\newcommand{\qp}{q_\p}
\newcommand{\f}{f}
\newcommand{\nbf}{\mathbf n}
\newcommand{\nbfOmega}{\nbf_{\Omega}}
\newcommand{\cmesh}{c_{\text{mesh}}}
\newcommand{\HotaunOmega}{H^1(\taun,\Omega)}
\newcommand{\Vn}{V_n}
\newcommand{\un}{u_n}
\let\div\relax
\DeclareMathOperator{\div}{div}

\newcommand{\vn}{v_n}
\newcommand{\vbfn}{\mathbf v_n}
\newcommand{\Lcal}{\mathcal L}
\newcommand{\zerobf}{\mathbf 0}
\newcommand{\ubf}{\mathbf u}
\newcommand{\ubftilde}{\widetilde{\mathbf u}}
\newcommand{\vbftilde}{\widetilde{\mathbf v}}
\newcommand{\fbf}{\mathbf f}
\newcommand{\vbf}{\mathbf v}
\newcommand{\Bn}{B_n}
\newcommand{\Dnt}{D_n^2}

\newcommand{\sigmabold}{\boldsymbol \sigma}
\newcommand{\taubold}{\boldsymbol \tau}
\newcommand{\csigmabold}{c_{\sigmabold}}
\newcommand{\ctaubold}{c_{\taubold}}
\newcommand{\cS}{c_S}
\newcommand{\cC}{c_C}

\DeclareMathOperator{\dG}{dG}

\newcommand{\Pcalpo}{\mathcal P_\p^1}
\newcommand{\Pcalpx}{\mathcal P_\p^x}

\newcommand{\Pcalpy}{\mathcal P_\p^y}
\newcommand{\Pcalpz}{\mathcal P_\p^z}
\newcommand{\Pcalp}{\mathcal P_\p}
\newcommand{\nbfE}{\mathbf n_\E}
\newcommand{\Pibfzp}{\boldsymbol \Pi^0_\p}
\newcommand{\Pizp}{\Pi^0_\p}
\newcommand{\Pizpmo}{\Pi^0_{\p-1}}
\newcommand{\Piop}{\Pi^1_\p}

\newcommand{\Pizxp}{\Pi^{0,x}_\p}
\newcommand{\Pizyp}{\Pi^{0,y}_\p}
\newcommand{\Pizzp}{\Pi^{0,z}_\p}

\newcommand{\Ibs}{\mathcal{I}_\p^{\rm BS}}
\newcommand{\Ihat}{\widehat I}
\newcommand{\Qhat}{\widehat Q}
\newcommand{\cod}{c_{1D}}
\newcommand{\up}{u_\p}
\newcommand{\partialx}{\partial_x}
\newcommand{\partialy}{\partial_y}
\newcommand{\partialz}{\partial_z}
\newcommand{\partialxy}{\partial_{xy}}

\newcommand{\Zbf}{\mathbf Z}

\title{$\h\p$-optimal interior penalty discontinuous Galerkin methods for the biharmonic problem}
\author{Z. Dong\thanks{Inria, 2 rue Simone Iff, 75589 Paris, France and CERMICS, Ecole des Ponts, 77455 Marne-la-Vall\'{e}e 2, France, {\tt zhaonan.dong@inria.fr}} ,
L. Mascotto\thanks{Dipartimento di Matematica e Applicazioni, Universit\`a di Milano Bicocca, 20125 Milan, Italy, {\tt lorenzo.mascotto@unimib.it}}
\thanks{Fakult\"at f\"ur Mathematik, Universit\"at Wien, 1090 Vienna, Austria, {\tt lorenzo.mascotto@univie.ac.at}}
\thanks{IMATI-CNR, 27100, Pavia, Italy}}
\date{}

\begin{document}

\maketitle

\begin{abstract}
\noindent We prove $\h\p$-optimal error estimates for interior penalty discontinuous Galerkin methods (IPDG) for the biharmonic problem
with homogeneous essential boundary conditions.
We consider tensor product-type meshes in two and three dimensions,
and triangular meshes in two dimensions.
An essential ingredient in the analysis is the construction
of a global $H^2$ piecewise polynomial approximants
with $\h\p$-optimal approximation properties over the given meshes.
The $\h\p$-optimality is also discussed
for $\mathcal C^0$-IPDG in two and three dimensions,
and the stream formulation of the Stokes problem in two dimensions.
Numerical experiments validate the theoretical predictions
and reveal that $\p$-suboptimality occurs
in presence of singular essential boundary conditions.


\end{abstract}

\section{Introduction} \label{section:introduction}

The numerical approximation of solutions to the biharmonic problem is a challenging task tracing back to the 60ties~\cite{Argyris-Fried-Scharpf:1968};
see also the later work~\cite{Douglas-Dupont-Percell-Scott:1979}.
The main difficulty in designing standard Galerkin methods for this problem resides in the construction of global $H^2$ conforming polynomial spaces
over the given meshes.
A possible way to circumvent this issue is based on using nonconforming finite element spaces.
The Morley element is the first and
one of the most popular among them;
see, e.g., \cite{Morley:1968, Ming-Xu:2006}.

Another option to avoid $H^2$ conformity is to use
$H^1$ conforming piecewise polynomial spaces.
This approach goes under the name of $\Ccal^0$-interior penalty discontinuous Galerkin method ($\Ccal^0$-IPDG)
and has been a very active area of research over the last~$20$ years;
see, e.g.,
\cite{Beirao-Niiranen-Stenberg:2007, Brenner-Gudi-Sung:2010,
Brenner-Sung:2005, Engel-Garikipati-Hughes-Larson-Mazzei-Taylor:2002}.
The continuity constraint can we be further weakened
leading to the use of fully discontinuous test and trial functions.
This approach goes under the name of interior penalty discontinuous Galerkin method (IPDG)
and was first investigated
for the biharmonic problem
by Baker~\cite{Baker:1977}
in the late 70ties;
amongst others, we recall the following contributions,
including the $\p$- and~$\h\p$-versions of the scheme:
\cite{Georgoulis-Houston:2009, Dong:2019, Dong-Mascotto-Sutton:2021, Gudi-Nataraj-Pani:2008, Mozolevski-Suli:2003,Suli-Mozolevski-Bosing:2007, Suli-Mozolevski:2007}.

Standard results in the theory of the IPDG for the biharmonic problem
state that the convergence is optimal with respect to the mesh size,
and one order and a half suboptimal with respect to the polynomial degree.
The reason for this is the use of polynomial inverse estimates
while handling the penalization of the bilinear form.
A milder suboptimality of half an order occurs also in the approximation
of elliptic partial differential equations of second order
but it can be removed under two assumptions:
(\emph{i}) the essential boundary conditions are homogeneous;
(\emph{ii}) the underlying mesh allows for the existence of a global~$H^1$ piecewise polynomial
with $\h\p$-optimal approximation properties.
This was first observed in~\cite{Perugia-Schotzau:2002}
using the error analysis in the $L^2$ norm
for the local discontinuous Galerkin method
and later analyzed in~\cite{Stamm-Wihler:2010}
using the error analysis in the energy norm for the IPDG.

The aim of this paper is to prove that an analogous result holds true using the IPDG for the biharmonic problem
with homogeneous essential boundary conditions.
We shall mainly focus on the Hessian-Hessian formulation in~\cite{Brenner-Sung:2005},
and show that similar $\h\p$-optimal results are valid also for
the $\mathcal{C}^0$-IPDG for the biharmonic problem in two and three dimensions,
and the stream formulation of the Stokes problem in two dimensions.
More precisely,
the main result of this contribution is that the optimality with respect to the polynomial degree
can be recovered using conforming tensor product-type meshes
in two and three dimensions,
and triangular meshes in two dimensions.

The analysis hinges upon showing the existence
of global $H^2$ piecewise polynomial $\h\p$-optimal approximants
over the given meshes.
On tensor product-type meshes,
we shall exhibit explicit $\h\p$-optimal bounds in the anisotropic Sobolev scale
that are fully explicit with respect to the Sobolev regularity index.
We are not aware of existing results in the literature covering the three dimensional case of this;
so, this topic is interesting \emph{per se}.
On triangular meshes, such an $H^2$ approximant
can be constructed following~\cite{Suri:1990}
but the error estimates depend on the Sobolev regularity index implicitly;
we are not aware of a counterpart of this approach on tetrahedral meshes.

If essential boundary conditions are nonhomogeneous,
for elliptic partial differential equations of second order,
it is known that the method is $\p$-suboptimal if essential boundary conditions are singular;
see~\cite{Georgoulis-Hall-Melenk:2010}.
A similar result might be proven for the biharmonic problem.
Showing this goes beyond the scopes of this work,
but we shall provide numerical evidence that indeed $\p$-suboptimality occurs.

\paragraph*{Structure of the paper.}
In the remainder of this section,
we present the notation we shall employ throughout the paper
and recall the Hessian-Hessian formulation of the biharmonic problem.
We design the corresponding IPDG scheme in Section~\ref{section:IPDG}.
Section~\ref{section:hp-polynomial} is devoted to recalling and proving several $\h\p$-polynomial approximation and stability properties on tensor product-type meshes,
which are instrumental to derive $\h\p$-optimal convergence rate of the IPDG in Section~\ref{section:convergence};
here, we also discuss $\h\p$-optimality on triangular meshes.
Variants of the above scheme are discussed in Section~\ref{section:variants},
namely we shall comment on $\mathcal C^0$-IPDG in two and three dimensions,
and the stream formulation of the Stokes problem in two dimensions.
Numerical experiments are presented in Section~\ref{section:numerical-experiments}
and conclusions are drawn in Section~\ref{section:conclusions}.

\paragraph*{Notation.}
Let~$D$ be a domain in~$\Rbb^d$, $d=1,2,3$.
We denote the Sobolev space of order~$s \in \Nbb$ by~$H^s(D)$.
Fractional order and negative order Sobolev spaces are defined using interpolation theory and duality, respectively.
We denote the Lebesgue space of square integrable functions over~$D$ by~$L^2(D)$.
We endow the above spaces with the inner product~$(\cdot,\cdot)_{s,D}$,
which induces the norm and seminorm~$\Norm{\cdot}_{s,D}$ and~$\SemiNorm{\cdot}_{s,D}$.
Further, given~$\ell \in \Rbb^+$ and~$k \in \Nbb$,
we introduce the anisotropic Sobolev space:
\begin{equation} \label{mixed-Sobolev-space}
\begin{split}
H^{\ell,k}(D)
:= \Big\{ v\in H^{\ell}(D) \Big| 
        & \Norm{\partial^{\alpha_1}_{x_1}
        \dots \partial^{\alpha_d}_{x_d} u }_{0,D} < \infty
        \quad \forall \boldsymbol \alpha
        = (\alpha_1, \dots,\alpha_d) \in \Nbb^d \text{ s. th.}\\
        & \vert \boldsymbol \alpha \vert = k,\ 
            \alpha_j\le 2 \ \forall j \in \{1,\dots,d\}
            \setminus \{J\}
            \text{ for one }J \in \{ 1,\dots, d\} \Big\}.
\end{split}
\end{equation}
The space~$H^{\ell,k}(D)$ contains $H^\ell$ functions
with square integrable mixed derivatives of order~$k$,
where such mixed derivatives have fewer
than three partial derivatives along all directions but one.

We introduce several weak differential operators:
$\nabla$ as the gradient;
$\Delta=\nabla\cdot\nabla$ as the Laplacian;
$\Delta^2$ as the bilaplacian;
$D^2=\nabla \nabla^T$ as the Hessian.

Given~$\p \in \Nbb$,
let~$\Pbb_\p(D)$ and~$\Qbb_\p(D)$ represent the spaces of polynomials of degree at most~$\p$ over a domain~$D$
and degree at most~$\p$ in each variable over a~$d$-dimensional cube~$D$, respectively.

We denote the scalar product between two vectors~$\ubf$ and~$\vbf$
in~$\Rbb^d$ by~$\ubf \cdot \vbf$.
Further, we denote the entrywise product between two tensors~$\ubf$ and~$\vbf$ in~$\Rbb^{d\times d}$ by~$\ubf : \vbf$.

Given two positive quantities~$a$ and~$b$,
we use the short-hand notation ``$a\lesssim b$''
if there exists a positive constant~$c$ independent of the discretization parameters
such that ``$a\le c \ b$''.
Moreover, we write ``$a \approx b$'' if and only if ``$a\lesssim b $''
and ``$b \lesssim a$''.

\paragraph*{The continuous problem.}
Given~$\Omega\subset\Rbb^d$, $d=2,3$, a bounded polygonal/polyhedral domain
with unit outward normal vector~$\nbfOmega$
and~$\f \in L^2(\Omega)$,
we consider the following biharmonic problem:
find~$u:\Omega \to \Rbb$ such that
\begin{equation} \label{strong-formulation}
\begin{cases}
\Delta^2 u = \f & \text{in } \Omega, \\
u=0 , \quad \nbfOmega \cdot \nabla u =0 & \text{on }\partial\Omega.
\end{cases}
\end{equation}
Introduce
\[
V:=H^2_0(\Omega),
\qquad\qquad
B(u,v)=\int_{\Omega} D^2 u : D^2 v
\quad \forall u,v\in V.
\]
A weak formulation of problem~\eqref{strong-formulation} reads
\begin{equation} \label{weak-formulation}
\begin{cases}
\text{find } u \in V \text{ such that}\\
B(u,v) = (\f,v)_{0,\Omega} \qquad \forall v\in V.
\end{cases}
\end{equation}
Problem~\eqref{weak-formulation} is well posed;
see, e.g., \cite[Section~5.9]{Brenner-Scott:2008}.

\begin{remark} \label{remark:nh-boundary-conditions}
We only consider homogeneous boundary conditions
as they allow for fewer cumbersome technicalities
when deriving $\h\p$-optimal error estimates with the IPDG.
However, the $\h\p$-optimal results that we shall discuss in this paper
can be extended to the case of piecewise smooth boundary conditions.
Moreover, handling natural boundary conditions does not hide additional technicalities
nor impact the optimal convergence rate in~$p$;
see, e.g., \cite{Beirao-Niiranen-Stenberg:2010}.
\eremk
\end{remark}

\section{Interior penalty discontinuous Galerkin methods for the biharmonic problem} \label{section:IPDG}
In this section, we introduce the IPDG for problem~\eqref{weak-formulation}.
More precisely, we introduce regular tensor product-type and triangular meshes in Section~\ref{subsection:meshes};
Section~\ref{subsection:method} is devoted to design the method;
in Section~\ref{subsection:properties-lifiting},
we discuss some properties of the bilinear forms involved in the method.

\subsection{Meshes} \label{subsection:meshes}
We consider sequences~$\{ \taun \}_{n}$ of decompositions of~$\Omega$ into parallelograms/parallelepipeds in two and three dimensions,
and triangular meshes in two dimensions.
Henceforth, we shall refer to the case of parallelograms/parallelepipeds meshes
as tensor product-type meshes.
We denote the set of ($d-1$)-dimensional faces of each~$\taun$ by~$\Fcaln$.
We further distinguish interior and boundary faces
introducing~$\FcalnI$ and~$\FcalnB$, respectively.
The set of vertices of~$\taun$ is~$\Vcaln$.

With each element~$\E \in \taun$, we associate its diameter~$\hE$
and collect all the diameters in~$\hbf:\Omega\to\Rbb^+$
defined as
\[
\hbf:=
\begin{cases}
\hE  & \text{if } \xbf \in \E \in \taun,\\
\hF  & \text{if } \xbf \in \F \in \Fcaln.\\
\end{cases}
\]
We denote the unit normal vector to any interior face~$F$
with respect to the global orientation of~$\F$ by~$\nbf$.
We always fix a outward pointing vector for boundary faces.

Throughout, we assume quasi-uniformity of the mesh~$\taun$, i.e.,
there exists~$\cmesh>0$ such that
\[
\cmesh^{-1} \h_{\E_1}
\le \h_{\E_2}
\le \cmesh \h_{\E_1}
\qquad \forall \E_1,\ \E_2 \in \taun.
\]
In what follows, we consider a uniform distribution of the polynomial degree~$\p\in\Nbb$ over~$\taun$.
Variable polynomial degrees can be used~\cite{Schwab:1997, Dong-Mascotto-Sutton:2021}
but are not the main focus of this work.

Define the scalar broken Sobolev space
\[
\begin{split}
& \HotaunOmega:=
\{ v \in L^2(\Omega) \mid v_{|\E} \in H^1(\E) \},\\
& H^{\ell,k}(\taun):=
\{ v \in L^2(\Omega) \mid v_{|\E} \in H^{\ell,k}(\E) \}
\quad \forall \ell \in \Rbb^+,\ k\in \Nbb.
\end{split}
\]
We introduce the jump~$\llbracket \cdot \rrbracket : L^2(\F) \to \Rbb$
and the average~$\{\cdot\} : L^2(\F) \to \Rbb$ operators
over a face~$\F \in \Fcaln$ as follows:
if the face~$\F \in \FcalnI$ is an interior face shared by the elements~$\Ep$ and~$\Em$,
given~$v\in\HotaunOmega$ with restrictions over the elements~$\Ep$ and~$\Em$ denoted by~$\vp$ and~$\vm$, respectively,
\[
\llbracket v \rrbracket_\F(\xbf)
=\llbracket v \rrbracket(\xbf)
:= \vp(\xbf) - \vm(\xbf),
\qquad
\{ v \}_\F(\xbf)
= \{ v \}(\xbf)
:= \frac12 (\vp(\xbf)+\vm(\xbf))
\qquad \forall \xbf \in \F.
\]
On boundary faces~$\F\in\FcalnB$, we set~$\llbracket v \rrbracket = \{ v \} =v$.

The definition of the two operators above can be extended to vector and tensor functions.
We avoid using different notation, for the presentation's sake.

For any differential operator introduced in Section~\ref{section:introduction},
we denote its elementwise version associated with~$\taun$
by adding a subscript~$n$;
for instance, $\nabla_n$ is the piecewise gradient over~$\taun$.

\subsection{The method} \label{subsection:method}
For each mesh~$\taun$ and polynomial degree~$\p$,
we introduce the space
\[
\Vn:=
\{ \vn \in L^2(\Omega) \mid
\vn{}_{|\E} \in \Pbb_\p(\E) \; \forall \E \in \taun  \}.
\]
We endow~$\Vn$ with the broken Sobolev norm and seminorm of order~$s>0$
\[
\Norm{\cdot}_{s,\taun}^2
:= \sum_{\E \in \taun} \Norm{\cdot}_{s,\E}^2,
\qquad
\SemiNorm{\cdot}_{s,\taun}^2
:= \sum_{\E \in \taun} \SemiNorm{\cdot}_{s,\E}^2,
\]
and introduce the lifting operator~$\Lcal:\Vn\oplus V \to [\Vn]^{d\times d}$, $d=2,3$, as
\begin{equation} \label{lifting-operator}
\int_{\Omega} \Lcal(\un) : \vbfn
:= \int_{\Fcaln} \{ \nbf \cdot (\nabla \cdot \vbfn) \} \llbracket \un \rrbracket
 - \int_{\Fcaln} \{ (\vbfn) \nbf \} \cdot \llbracket \nabla \un \rrbracket
\qquad  \forall \vbfn \in [\Vn]^{d\times d}.
\end{equation}
Let~$\sigmabold:\Fcaln \to \Rbb^+$ and~$\taubold:\Fcaln\to\Rbb^+$
two piecewise constant stabilization parameters, which we shall fix in what follows; see Section~\ref{subsection:properties-lifiting} below.
Given~$\Dnt$ the piecewise Hessian over~$\taun$,
we define the bilinear form~$\Bn:\Vn\times\Vn\to \Rbb$ as
\begin{equation} \label{bilinear-form}
\begin{split}
\Bn(\un,\vn)
& := \int_\Omega \Dnt \un : \Dnt \vn
  + \int_{\Omega} \left( \Lcal(\un):\Dnt\vn + \Lcal(\vn) :\Dnt \un    \right)\\
& \quad  + \int_{\Fcaln}
    \left( \sigmabold \llbracket \un \rrbracket  \llbracket \vn \rrbracket
    + \taubold \llbracket \nabla\un \rrbracket \cdot \llbracket \nabla\vn \rrbracket \right)
\qquad \forall \un,\vn\in\Vn.
\end{split}
\end{equation}
We consider the following IPDG for the approximation of solutions to~\eqref{weak-formulation}; see~\cite{Dong-Mascotto-Sutton:2021}:
\begin{equation} \label{dG}
\begin{cases}
\text{find } \un \in \Vn \text{ such that}\\
\Bn(\un,\vn) = (f,\vn) \qquad \forall \vn \in \Vn.
\end{cases}
\end{equation}
Other dG formulations might be used~\cite{Brenner-Sung:2005,Georgoulis-Houston:2009};
yet, we stick to~\eqref{dG} as it is the most similar to that of~\cite{Stamm-Wihler:2010}.

Eventually, we define the natural norm associated with the bilinear form~$\Bn(\cdot,\cdot)$:
for all~$v \in \Vn \oplus V$,
\begin{equation} \label{dg-norm}
\begin{split}
\Norm{v}_{\dG}^2
& := \Norm{\Dnt v}_{0,\Omega}^2
   + \sum_{\F \in \Fcaln} \Norm{\taubold^{\frac12}\llbracket \nabla v \rrbracket }_{0,\F}^2
   + \sum_{\F \in \Fcaln} \Norm{\sigmabold^{\frac12} \llbracket v \rrbracket}_{0,\F}^2\\
& =: \Norm{\Dnt v}_{0,\Omega}^2
   + \Norm{\taubold^{\frac12}\llbracket \nabla v \rrbracket }_{0,\Fcaln}^2
   + \Norm{\sigmabold^{\frac12} \llbracket v \rrbracket}_{0,\Fcaln}^2,
\end{split}
\end{equation}
where
\[
(u,v)_{0,\Fcaln}
:= \sum_{\F \in \Fcaln} (u , v )_{0,\F},
\qquad
\Norm{u}_{0,\Fcaln}^2 := (u,u)_{0,\Fcaln}
\qquad\qquad \forall u,v \in \Vn \oplus V.
\]

\subsection{Properties of the lifting operator, bilinear form, and right-hand side} \label{subsection:properties-lifiting}
Following~\cite[Lemma~5.1]{Georgoulis-Houston:2009},
we have the following continuity result for the lifting operator~$\Lcal(\cdot)$ in~\eqref{lifting-operator}.
\begin{lemma} \label{lemma:lifting-properties}
There exists a positive constant~$\cS$ independent of~$\hbf$ and~$\p$
such that
\[
\Norm{\Lcal(\vn)}_{0,\Omega}
\le \cS
\left( \Norm{\left(\frac{\p^6}{\hbf^3}\right)^\frac12 \llbracket \vn \rrbracket}_{0,\Fcaln}
+ \Norm{\left(\frac{\p^2}{\hbf}\right)^\frac12 \llbracket \nabla \vn \rrbracket}_{0,\Fcaln} \right)
\qquad \forall \vn \in \Vn.
\]
\end{lemma}

Consequences of Lemma~$\ref{lemma:lifting-properties}$ are the coercivity and  continuity of the dG bilinear form~$\Bn(\cdot,\cdot)$
under suitable choices of the dG parameters~$\sigmabold$ and~$\taubold$,
namely, given~$\csigmabold$, $\ctaubold>2\cS+\frac12$,
\begin{equation} \label{dg-parameters}
\sigmabold:= \csigmabold \frac{p^6}{\hbf^3},
\qquad
\taubold:=\ctaubold \frac{\p^2}{\hbf}.
\end{equation}
See, e.g., \cite[Lemma~5.2]{Georgoulis-Houston:2009}.

\begin{lemma} \label{lemma:properties-bf}
The bilinear form~$\Bn(\cdot,\cdot)$ is coercive and continuous:
\begin{equation} \label{continuity&coercivity:bf}
\Bn(v,v)\ge \frac12 \Norm{v}_{\dG}^2,
\qquad
\Bn(u,v) \le 2 \Norm{u}_{\dG} \Norm{v}_{\dG}
\qquad \qquad \forall u,\ v \in \Vn\oplus V.
\end{equation}
\end{lemma}
We also have the continuity of the right-hand side in~\eqref{dG}.
\begin{lemma} \label{lemma:rhs}
There exists a positive constant~$\cC$ independent of~$\hbf$ and~$\p$
such that
\[
(\f,v)_{0,\Omega}
\le \cC \Norm{\f}_{0,\Omega}
    \Norm{v}_{\dG}
\qquad \forall v\in \Vn \oplus V.
\]
\end{lemma}
\begin{proof}
The assertion follows from the Cauchy-Schwarz inequality
and Poincar\'e-Friedrichs type inequalities for piecewise~$H^2$ functions;
see, e.g., \cite[Corollary~$4.2$]{Brenner-Wang-Zhao:2004}.
\end{proof}
The well posedness of method~\eqref{dG} follows combining 
the coercivity of~$\Bn(\cdot,\cdot)$ and the continuity of~$(f,\cdot)_{0,\Omega}$
with respect to the~$\Norm{\cdot}_{\dG}$ norm.

\begin{theorem} \label{theorem:well-posedness-dG}
Method~\eqref{dG} is well posed with a priori estimates
\[
\Norm{\un}_{\dG} \le 2\cC \Norm{\f}_{0,\Omega}.
\]
\end{theorem}

\section{Polynomial approximation results on tensor product-type meshes} \label{section:hp-polynomial}

In this section, we present several $\h\p$-optimal
polynomial approximation results.
They will be useful in Section~\ref{section:convergence} below
when proving the convergence rate of the IPDG
on tensor product-type meshes.

\paragraph*{Preliminary technical results.}
We recall a result concerned with approximation properties of the~$L^2$ and~$H^1$ projectors in one dimension;
see, e.g., \cite{Schwab:1997, Canuto-Quarteroni:1982} for details.
Given~$\Ihat:=(-1,1)$,
let~$\Pizp : L^2(\Ihat) \to \Pbb_\p(\Ihat)$
and~$\Piop : H^1(\Ihat) \to \Pbb_\p(\Ihat)$
be the $L^2$ and~$H^1$ projectors onto~$\Pbb_\p(\Ihat)$, respectively:
for all~$\qp \in \Pbb_\p(\Ihat)$,
\[
(\qp, v-\Pizp v)_{0,\Ihat} =0 \quad \forall v \in L^2(\Ihat),
\qquad
\begin{cases}
(\qp', (v-\Piop v)')_{0,\Ihat} =0 \\
(v - \Piop v)(\pm 1)=0
\end{cases}
\quad \forall v \in H^1(\Ihat).
\]

The following approximation estimates are valid;
see, e.g., \cite[Chapter 3]{Schwab:1997}
and \cite[Lemma~3.5]{Georgoulis:2003}.
\begin{lemma} \label{lemma:L2-H1:projectors}
For all~$\p \in \Nbb$,
given~$u \in H^{k}(\Ihat)$
with $k\geq 1$ and~$s:=\min(k,\p+1)$,
there exists a positive constant~$c$ independent of~$\p$ and~$k$
such that
\begin{equation} \label{H1-L2:estimates-CQ}
\begin{split}
\Norm{u-\Pizp u}_{0,\Ihat} & \le c \p^{-s} \SemiNorm{u}_{s,\Ihat},\\
\Norm{(u-\Piop u)'}_{0,\Ihat} & \le c \p^{-s+1} \SemiNorm{u}_{s,\Ihat},\\
\Norm{(u-\Pizp u)'}_{0,\Ihat} & \le c \p^{-s+\frac32} \SemiNorm{u}_{s,\Ihat}.
\end{split}
\end{equation}

\end{lemma}

Next, we recall best approximation estimates of the elementwise $L^2$ projector~$\Pizp:L^2(\E) \to \Qbb_\p(\E)$
with respect to the boundary $L^2$ norm
on a $d$-dimensional tensor product-element $\E$;
see, e.g., \cite[Lemma~$3.9$]{Houston-Schwab-Suli:2002}.
\begin{lemma} \label{lemma:L2bulk:L2edge}
Let~$\E \subset \Rbb^d$, $d=2,3$ be a tensor product-type element,
i.e., a parallelogram if~$d=2$ and a parallelepiped if~$d=3$,
$\hE$ its diameter,
and~$u \in H^k(\E)$, $k\ge 1$.
Then, for all~$1\le s \le \min (\p+1,k)$, $\p \in \Nbb$,
there exists a positive constant~$c>0$
independent of~$\hE$, $k$, and~$\p$ such that
\[
\Norm{u-\Pizp u}_{0,\partial\E}
\le c \left( \frac{\hE}{\p+1} \right)^{s - \frac12}
        \Norm{u}_{s, \E}.
\]
\end{lemma}

\noindent Now, we discuss $\h\p$-optimal approximation properties
by means of global $H^2$ piecewise polynomials over~$\taun$.
We proceed in three steps.

\paragraph*{$H^2$ estimates in one dimension.}
We recall the existence of a polynomial matching the values of a given function and its derivative
at the endpoints of an interval with $\h\p$-optimal
approximation properties;
see~\cite[Corollary~$2$]{Beirao-Buffa-Rivas-Sangalli:2011}.
\begin{lemma} \label{lemma:H2-1D}
Let~$\Ihat:=(-1,1)$ and~$k\ge0$.
Then, for all~$\p\ge3$,
there exists a projection operator~$\Pcalpo: H^{k+2}(\Ihat) \to \Pbb_\p(\Ihat)$
such that the following continuity properties are valid:
for all~$u\in H^{k+2}(\Ihat)$
\begin{equation} \label{continuity:function&derivative}
(\Pcalpo u) (\pm 1) = u(\pm 1),
\qquad\qquad
(\Pcalpo u)' (\pm 1) = u'(\pm 1) .
\end{equation}
Moreover, the following upper bounds are valid:
given~$s := \min(k,\p-1)$,
\begin{equation} \label{1D:hp-approx}
\begin{split}
\Norm{(u-\Pcalpo u)''}_{0,\Ihat}^2
& \le \frac{(\p-s-1)!}{(\p+s-1)!}
        \Norm{u^{(s+2)}}_{0,\Ihat}^2, \\
\Norm{(u-\Pcalpo u)'}_{0,\Ihat}^2
& \le \frac{(\p-s-1)!}{(\p+s-1)!} \frac{1}{(\p-1)\p}
            \Norm{u^{(s+2)}}_{0,\Ihat}^2,\\
\Norm{u-\Pcalpo u}_{0,\Ihat}^2
& \le \frac{(\p-s-1)!}{(\p+s-1)!}
    c(\p)  \Norm{u^{(s+2)}}_{0,\Ihat}^2,
\end{split}
\end{equation}
where $c(\p)$ is independent of~$k$ and
\[
\begin{split}
c(\p)\approx \p^{-4}.
\end{split}
\]
\end{lemma}
\begin{proof}
We provide the details of the proof in Appendix~\ref{appendix:lemma:H2-1D} for completeness.
\end{proof}
An application of Stirling's formula,
see, e.g., \cite[Corollary~$3.12$]{Schwab:1997},
yields the following $\p$-optimal approximation result
for the projector~$\Pcalpo$.
\begin{corollary} \label{corollary:H2-1D}
Let~$\Ihat:=(-1,1)$ and~$\p\ge3$.
Then, there exists a positive constant~$\cod>0$ independent of~$\p$ and~$k$ such that
the projection operator~$\Pcalpo: H^{k+2}(\Ihat) \to \Pbb_\p(\Ihat)$, $k \ge 0$, from Lemma~\ref{lemma:H2-1D} satisfies,
for $s:=\min(k,\p-1)$,
\begin{equation} \label{1D:hp-approx:cor}
\begin{split}
& \cod (\p-1)^2 \Norm{u-\Pcalpo u}_{0,\Ihat}
+ (\p-1) \Norm{(u-\Pcalpo u)'}_{0,\Ihat}
+ \Norm{(u-\Pcalpo u)''}_{0,\Ihat} \\
& \le \left(\frac{e}{2}\right)^{\frac{s}{\p-1}}
    (\p-1)^{-s} \Norm{u^{(s+2)}}_{0,\Ihat}
  \le \frac{e}{2}
    (\p-1)^{-s} \Norm{u^{(s+2)}}_{0,\Ihat}.
\end{split}
\end{equation}
\end{corollary}

Based on this, we derive stability properties for the projection operator constructed in Lemma~\ref{lemma:H2-1D}.
\begin{corollary} \label{corollary:H2-stability-1D}
The projection operator~$\Pcalpo \cdot$ constructed in Lemma~\ref{lemma:H2-1D}
satisfies the following stability properties:
there exists a positive constant~$c$ independent
of~$\p$ and~$k$ such that,
for all~$u\in H^2(\Ihat)$,
\begin{equation} \label{stability:Pcalp}
\begin{split}
\Norm{(\Pcalpo u)''}_{0,\Ihat}
& \le \Norm{u''}_{0,\Ihat},\\
\Norm{(\Pcalpo u)'}_{0,\Ihat}
& \le \Norm{u'}_{0,\Ihat} + \p^{-1} \Norm{u''}_{0,\Ihat},\\
\Norm{\Pcalpo u}_{0,\Ihat}
& \le  \Norm{u}_{0,\Ihat} + c \p^{-2} \Norm{u''}_{0,\Ihat}.
\end{split}
\end{equation}
\end{corollary}
\begin{proof}
The first inequality follows from the definition of the operator~$\Pcalpo$ in~\eqref{writing:u:1}
and the orthogonality properties of Legendre polynomials.
As for the other two inequalities,
it suffices to add and subtract~$u'$ and~$u$,
respectively, use the triangle inequality,
and apply the estimates in~\eqref{1D:hp-approx} with~$s=0$.
\end{proof}

\paragraph*{$H^2$ estimates on tensor product elements.}
Given~$\Pcalpx$, $\Pcalpy$, and~$\Pcalpz$ the one dimensional projectors from Lemma~\ref{lemma:H2-1D} applied to the restrictions of functions along the directions~$x$, $y$, and~$z$, respectively,
we define the two and three dimensional operators
\[
\Pcalp \cdot := \Pcalpx \Pcalpy \cdot,
\qquad\qquad
\Pcalp \cdot := \Pcalpx \Pcalpy \Pcalpz \cdot.
\]
By tensor product arguments,
we deduce the two and three dimensional counterparts of Lemma~\ref{lemma:H2-1D}.

\begin{theorem} \label{theorem:H2-2D/3D}
Let~$\Qhat:=(-1,1)^2$ be the reference square,
with vertical edges~$\e_1$ and~$\e_3$, horizontal edges~$\e_2$ and~$\e_4$,
and outward unit normal vector~$\nbf_{\Qhat}$.
Further, let~$\p\ge3$
and~$u \in H^{2,k+2}(\Qhat)$, $k \ge 0$.
Then, the following continuity properties on~$\partial \Qhat$ are valid:
\begin{equation} \label{continuity:2D}
\Pcalp u_{|\e_i}=
\begin{cases}
\Pcalpx (u_{|\e_i}) & \text{if $i$ is even,}\\
\Pcalpy (u_{|\e_i}) & \text{if $i$ is odd,}
\end{cases}
\qquad
(\nbf_{\Qhat} \cdot \nabla \Pcalp u)_{|\e_i}=
\begin{cases}
\nbf_{\Qhat} \cdot \nabla \Pcalpy (u_{|\e_i}) & \text{if $i$ is even,}\\
\nbf_{\Qhat} \cdot \nabla \Pcalpx (u_{|\e_i}) & \text{if $i$ is odd.}
\end{cases}
\end{equation}
In particular, $\Pcalp u$ and~$u$ share the same value at the vertices of~$\Qhat$.
If~$\Qhat=(-1,1)^3$, analogous continuity properties are valid.

Let~$s = \min(k,\p-1)$.
Then, the following optimal $\p$-error estimates are valid.
If~$\Qhat=(-1,1)^2$,
there exists a positive constant~$c$ independent
of~$\p$ and~$k$ such that
\begin{align} \label{2D:hp-approx:L2}
\Norm{u-\Pcalp u}_{0,\Qhat}
&\le c \p^{-s-2} \left( \Norm{\partialx^{s+2}u}_{0,\Ihat}
    + \Norm{\partialy^{s+2}u}_{0,\Ihat}
    + \Norm{\partialx^{2} \partialy^{s} u}_{0,\Ihat} \right);\\
 \label{2D:hp-approx:H1}
\Norm{\nabla (u-\Pcalp u)}_{0,\Qhat}
& \le c \p^{-s-1}
\left( \Norm{\partialx^{s+2}u}_{0,\Qhat}
         + \Norm{\partialx \partialy^{s+1}u}_{0,\Qhat}
         + \Norm{\partialx^2 \partialy^{s} u}_{0,\Qhat}\right.\\
& \nonumber
\qquad \qquad \quad  \left. + \Norm{\partialx^{s} \partialy^2 u}_{0,\Qhat}
        + \Norm{\partialx^{s+1} \partialy u}_{0,\Qhat}
        + \Norm{\partialy^{s+2}u}_{0,\Qhat} \right);\\
\label{2D:hp-approx:H2}
\Norm{D^2 (u-\Pcalp u)}_{0,\Qhat}
&\le c \p^{-s}
\left( \Norm{\partialx^{s+2} u}_{0,\Qhat}
        + \Norm{\partialx^2 \partialy^{s} u}_{0,\Qhat}
        + \Norm{\partialy^{s+2} u}_{0,\Qhat} \right.\\
& \nonumber
\qquad \qquad \left. + \Norm{\partialx^{s} \partialy^2 u}_{0,\Qhat}
         \Norm{\partialx^{s+1} \partialy u}_{0,\Qhat}
        + \Norm{\partialx \partialy^{s+1}u}_{0,\Qhat}
        + \Norm{\partialx^2 \partialy^{s} u}_{0,\Qhat} \right).
\end{align}
Instead, if~$\Qhat=(-1,1)^3$,
there exists a positive constant~$c$
independent of~$\p$ and~$k$ such that
\begin{align}
\label{3D:hp-approx:L2}
\Norm{u-\Pcalp u}_{0,\Qhat}
&\le c \p^{-s-2}
    \left( \Norm{\partialx^{s+2}u}_{0,\Ihat} + \Norm{\partialy^{s+2}u}_{0,\Ihat} + \Norm{\partialz^{s+2}u}_{0,\Ihat}\right.\\
& \nonumber
\quad   \left. + \Norm{\partialx^{2} \partialz^{s} u}_{0,\Ihat}
    + \Norm{\partialy^{2} \partialz^{s} u}_{0,\Ihat}
    + \Norm{\partialx^{2} \partialy^{2} \partialz^{s-2} u}_{0,\Ihat}
    \right);\\
\label{3D:hp-approx:H1}
\Norm{\nabla (u-\Pcalp u)}_{0,\Qhat}
& \le c \p^{-s-1}
 \left( \Norm{\partialx^{s+2}u}_{0,\Qhat}
        + \Norm{\partialy^{s+2}u}_{0,\Qhat}
        + \Norm{\partialz^{s+2}u}_{0,\Qhat}
        + \Norm{\partialx \partialy^{s+1}u}_{0,\Qhat}\right.\\
& \nonumber \hspace*{-2cm}
\;     + \Norm{\partialy \partialz^{s+1}u}_{0,\Qhat}
        + \Norm{\partialz \partialx^{s+1}u}_{0,\Qhat}
        + \Norm{\partialx \partialy^{s+1}u}_{0,\Qhat}
        + \Norm{\partialy \partialz^{s+1}u}_{0,\Qhat}
        + \Norm{\partialz \partialx^{s+1}u}_{0,\Qhat}\\
& \nonumber \hspace*{-2cm}
\;     + \Norm{\partialx \partialy^2 \partialz^{s-1}u}_{0,\Qhat}
        + \Norm{\partialy \partialz^2 \partialx^{s-1}u}_{0,\Qhat}
        + \Norm{\partialz \partialx^2 \partialy^{s-1}u}_{0,\Qhat}
        + \Norm{\partialx^2 \partialy^{s}u}_{0,\Qhat}
        + \Norm{\partialy^2 \partialz^{s}u}_{0,\Qhat} \\
& \nonumber \hspace*{-2cm}
\;     + \Norm{\partialz^2 \partialx^{s}u}_{0,\Qhat}
        + \Norm{\partialx^2 \partialz^{s}u}_{0,\Qhat}
        + \Norm{\partialy^2 \partialx^{s}u}_{0,\Qhat}
        + \Norm{\partialz^2 \partialy^{s}u}_{0,\Qhat}\\
& \nonumber  \hspace*{-2cm}
\;     \left.   + \Norm{\partialx^2 \partialy^2 \partialz^{s-2}u}_{0,\Qhat}
        + \Norm{\partialy^2 \partialz^2 \partialx^{s-2}u}_{0,\Qhat}
        + \Norm{\partialz^2 \partialx^2 \partialy^{s-2}u}_{0,\Qhat} \right);\\
\label{3D:hp-approx:H2}
\Norm{D^2 (u-\Pcalp u)}_{0,\Qhat}
& \le c \p^{-s}
\left( \Norm{\partialx^{s+2} u}_{0,\Qhat}
        + \Norm{\partialy^{s+2} u}_{0,\Qhat}
        + \Norm{\partialz^{s+2} u}_{0,\Qhat}
        + \Norm{\partialx^{2} \partialy^{s} u}_{0,\Qhat}
     \right.\\
& \nonumber \hspace*{-2cm}
\; + \Norm{\partialy^{2} \partialz^{s} u}_{0,\Qhat}
     + \Norm{\partialz^{2} \partialx^{s} u}_{0,\Qhat}
     + \Norm{\partialx^{2} \partialz^{s} u}_{0,\Qhat}
     + \Norm{\partialy^{2} \partialx^{s} u}_{0,\Qhat}
     + \Norm{\partialz^{2} \partialy^{s} u}_{0,\Qhat}
     + \Norm{\partialx^{s+1} \partialy u}_{0,\Qhat}\\
& \nonumber \hspace*{-2cm}
\; + \Norm{\partialy^{s+1} \partialz u}_{0,\Qhat}
     + \Norm{\partialz^{s+1} \partialx u}_{0,\Qhat}
     + \Norm{\partialx \partialy^{s+1} u}_{0,\Qhat}
     + \Norm{\partialy \partialz^{s+1} u}_{0,\Qhat}
     + \Norm{\partialz \partialx^{s+1} u}_{0,\Qhat}\\
& \nonumber \hspace*{-2cm}
\; + \Norm{\partialx \partialy \partialz^{s} u}_{0,\Qhat}
     + \Norm{\partialy \partialz \partialx^{s} u}_{0,\Qhat}
     + \Norm{\partialz \partialx \partialy^{s} u}_{0,\Qhat}
     + \Norm{\partialx \partialy^2 \partialz^{s-1} u}_{0,\Qhat}
     + \Norm{\partialy \partialz^2 \partialx^{s-1} u}_{0,\Qhat}\\
& \nonumber \hspace*{-2cm}
\; + \Norm{\partialz \partialx^2 \partialy^{s-1} u}_{0,\Qhat}
     + \Norm{\partialx^{2} \partialy \partialz^{s-1} u}_{0,\Qhat}
     + \Norm{\partialy^{2} \partialz \partialx^{s-1} u}_{0,\Qhat}
     + \Norm{\partialz^{2} \partialx \partialy^{s-1} u}_{0,\Qhat}\\
& \nonumber \hspace*{-2cm}
\; \left. + \Norm{\partialx^2 \partialy^2 \partialz^{s-2} u}_{0,\Qhat}
            + \Norm{\partialy^2 \partialz^2 \partialx^{s-2} u}_{0,\Qhat}
            + \Norm{\partialz^2 \partialx^2 \partialy^{s-2} u}_{0,\Qhat} \right).
\end{align}
\end{theorem}
\begin{proof}
As for the two dimensional case, we refer to~\cite{Beirao-Buffa-Rivas-Sangalli:2011}
and present some details in Appendix~\ref{appendix:theorem:H2-2D} for completeness.
We report the proof for the three dimensional estimates in Appendix~\ref{appendix:theorem:H2-3D}.
\end{proof}
\begin{remark} \label{remark:mixed-order Sobolev spaces}
In Theorem~\ref{theorem:H2-2D/3D}, the $p$-explicit approximation bounds are derived
for functions in the anisotropic Sobolev space $H^{2,k+2}(\Qhat)$
defined in~\eqref{mixed-Sobolev-space}.
This is due to the use of tensor product arguments based
on one dimensional approximation results.
The upside of considering such a regularity
is that the error bounds are fully explicit
with respect to the polynomial degree~$p$
and the Sobolev regularity index~$k$;
this is crucial to derive exponential convergence rate of the $hp$-version Galerkin methods.
Moreover, estimates as in Theorem~\ref{theorem:H2-2D/3D}
naturally appear in the approximation of solutions
to partial differential equations
with anisotropic Sobolev regularity.
We cannot extend the techniques derived so far
to the case of triangular meshes.
Notably, in that case, the estimates
also depend on the Sobolev regularity index~$k$;
see Section~\ref{subsection:optimal-convergence-triangle} below for more details.
\eremk
\end{remark}

\begin{remark} \label{remark:anisotropy}
In the error estimates of Theorem~\ref{theorem:H2-2D/3D},
it is possible to switch the order of the mixed derivatives
appearing in the norms on the right-hand sides.
This is of paramount importance to notice
in case one wishes to approximate
functions in anisotropic Sobolev spaces.
\eremk
\end{remark}

So far, we proved error estimates in two and three dimensions
for tensor product-type meshes.
These results can be extended to higher dimensions.
The corresponding proof hinges upon the same arguments used in Theorem~\ref{theorem:H2-2D/3D}:
error estimates in one dimension;
stability of the $H^2$ projector~$\Pcalpo$;
identities as in~\eqref{essential-commutative-property:2D} in general dimension.


By a scaling argument, we deduce $\h\p$-optimal error estimates
for functions with isotropic Sobolev regularity on tensor product-type elements.
\begin{corollary} \label{corollary:H2-2D}
Let~$\E$ be any parallelogram
with vertical edges~$\e_1$ and~$\e_3$, horizontal edges~$\e_2$ and~$\e_4$,
and outward unit normal vector~$\nbfE$, $\p\ge3$,
and~$u \in H^{k+2}(\E)$, $k \ge 0$.
Then, the following continuity properties on~$\partial \E$ are valid:
\[
\Pcalp u_{|\F_i}=
\begin{cases}
\Pcalpx (u_{|\F_i}) & \text{if $i$ is even,}\\
\Pcalpy (u_{|\F_i}) & \text{if $i$ is odd,}
\end{cases}
\qquad
(\nbfE \cdot \nabla \Pcalp u)_{|\F_i}=
\begin{cases}
\nbfE \cdot \nabla \Pcalpy (u_{|\F_i}) & \text{if $i$ is even,}\\
\nbfE \cdot \nabla \Pcalpx (u_{|\F_i}) & \text{if $i$ is odd.}
\end{cases}
\]
In particular, $\Pcalp u$ and~$u$ share the same value at the vertices of~$\E$.
If~$\E$ is a parallelepipeds, analogous continuity properties hold true.

Besides, the following optimal $\p$-error estimates are valid in two and three dimensions:
given~$s=\min(k,\p-1)$
and either a parallelogram or parallelepiped~$\E$,
there exists a positive constant~$c$
independent of~$\h$, $k$, and~$\p$ such that
\begin{equation} \label{hp-H2-approximation}
\frac{\p^2}{\hE^2} \Norm{u-\Pcalp u}_{0,\E}
+ \frac{\p}{\hE} \Norm{\nabla (u-\Pcalp u)}_{0,\E}
+ \Norm{D^2(u-\Pcalp u)}_{0,\E}
\le c
\frac{\h^{s}}{\p^{s}} \SemiNorm{u}_{s+2,\E}.
\end{equation}
\end{corollary}

\paragraph*{Global $H^2$ estimates.}
Corollary~\ref{corollary:H2-2D} implies $\h\p$-optimal error estimates by means of
global~$H^2$ piecewise polynomials over the given tensor product-type meshes.
More precisely, we have the following result.
\begin{theorem} \label{theorem:approximation-tensor-product-meshes}
Let~$\taun$ be a tensor product-type mesh
over~$\Omega \subset \Rbb^d$, $d=2,3$,
of either parallelograms ($d=2$) or parallelepipeds ($d=3$)
as in Section~\ref{subsection:meshes};
$k\ge0$; $\p \in \Nbb$; $s:= \min(\p,k-1)$.
Then there exists $\Pcalp: H^2(\Omega) \cap H^{2,2d}(\taun) \to H^2(\Omega) \cap \Vn$
and a positive constant~$c$ independent of~$\hbf$, $k$, and~$\p$ such that,
for all~$u \in H^2(\Omega) \cap H^{k+2}(\taun)$,
\[
\p^2 \Norm{\hbf^{-2}        (u-\Pcalp u)}_{0,\Omega}
+ \p \Norm{\hbf^{-1} \nabla (u-\Pcalp u)}_{0,\Omega}
+    \Norm{          D^2    (u-\Pcalp u)}_{0,\Omega}
\le
c \p^{-s} \Norm{\hbf^s D^{s+2} u}_{0,\taun}.
\]
\end{theorem}

The operator~$\Pcalp$ is defined on the space~$H^2(\Omega) \cap H^{2,2d}(\taun)$.
The reason for this will be apparent from the construction of the operator itself
in Appendices~\ref{appendix:lemma:H2-1D},
\ref{appendix:theorem:H2-2D},
and~\ref{appendix:theorem:H2-3D}.
In particular, we shall need that mixed derivatives of order~$2$ in each direction
are square integrable.

\section{Optimal convergence} \label{section:convergence}

In this section, we show the main result of this paper:
method~\eqref{dG} is $\h\p$-optimal on different types of meshes
if the biharmonic problem~\eqref{weak-formulation} is endowed with homogeneous boundary conditions.
We focus on tensor product-type meshes in 2D and 3D in Section~\ref{subsection:optimal-tensor-product};
estimates in weaker norms are shown in Section~\ref{subsection:Aubin:Nitsche};
a guideline on how to show $\h\p$-optimal error estimates on triangular meshes
is provided in Section~\ref{subsection:optimal-convergence-triangle}.

\subsection{Optimal convergence on tensor product-type meshes}
\label{subsection:optimal-tensor-product}

In this section, we prove $\h\p$-optimal error estimates for tensor product-type meshes in two and three dimensions.
The analysis bases on the existence of the global $H^2$ polynomial projector from Theorem~\ref{theorem:approximation-tensor-product-meshes}.

Henceforth, given~$D \subset \Rbb^d$, $d=2,3$,
either a parallelogram (in two dimensions)
or a parallelepiped (in three dimensions),
we introduce~$\Pizp$ as the composition of the one dimensional $L^2$ projectors in Lemma~\ref{lemma:L2-H1:projectors}
along the coordinate directions.
In other words,
we let
\[
\Pizp = \Pizxp \Pizyp \quad\text{ if }\quad d=2,
\qquad\qquad
\Pizp = \Pizxp \Pizyp \Pizzp \quad\text{ if }\quad d=3.
\]

\begin{theorem} \label{theorem:convergence}
Let~$\taun$ be a quasi-uniform tensor product-type mesh of either parallelograms (2D) or parallelepipeds (3D)
as in Section~\ref{subsection:meshes}.
Let~$u$ and~$\un$ be the solutions to~\eqref{weak-formulation} and~\eqref{dG}, respectively.
Assume that~$u\in H^{s+2}(\Omega)$, $s:=\min(k,\p-1)$, $\p \ge 2$.
Then, there exists a positive constant~$c$
independent of~$\h$, $k$, and~$\p$
such that
\[
\Norm{u-\un}_{\dG}
\le c \left(\frac{\h}{\p}\right)^{s} \Norm{u}_{s+2,\Omega}.
\]
In other words, for homogeneous essential boundary conditions,
method~\eqref{dG} is $\h\p$-optimal.
\end{theorem}
\begin{proof}
As in~\cite[Theorem~$5.3$]{Georgoulis-Houston:2009},
we have the following error splitting:
\begin{equation} \label{error-splitting}
\Norm{u-\un}_{\dG}
\le 5 \inf_{\vn \in \Vn} \Norm{u-\vn}_{\dG}
    + 2 \sup_{\xi \in \Vn} \frac{\vert\Bn(u-\un,\xi)\vert}{\Norm{\xi}_{\dG}}.
\end{equation}
The first term represents a best polynomial approximation result;
the second measures the inconsistency of the scheme.

We have standard $\h\p$-optimal (piecewise) polynomial estimates on the first term:
\begin{equation} \label{first-term-Strang}
\inf_{\vn \in \Vn} \Norm{u-\vn}_{\dG}
\lesssim \left(\frac{\h}{\p}\right)^{s} \Norm{u}_{s+2,\Omega}.
\end{equation}
To show~\eqref{first-term-Strang}, it suffices to pick~$\vn = \Pcalp u$,
where~$\Pcalp\cdot$ is as in Theorem~\ref{theorem:approximation-tensor-product-meshes}.
The jump terms on interior edges vanish
due to the global $H^2$ regularity of~$u$ and~$\Pcalp u$;
the jump terms on boundary edges vanish since~$u$ has homogeneous essential boundary conditions
and~$\Pcalp u$ preserves polynomial essential boundary conditions.

Thus, we only need to derive estimates for the inconsistency term,
i.e., the second term on the right-hand side of~\eqref{error-splitting}.
Proceeding as in~\cite[Lemma~$5.4$]{Georgoulis-Houston:2009},
we set~$\eta:=u-\un$ and write
\begin{equation} \label{starting-point}
\begin{split}
\Bn(u-\un,\xi)
& = \int_{\Fcaln} 
    \big\{ \nbf \cdot (\nabla\cdot \Dnt \eta)
    - \nbf \cdot (\nabla\cdot \Pibfzp \Dnt \eta) \big\}
    \llbracket \xi \rrbracket\\
& - \int_{\Fcaln} \big\{ (\Dnt \eta)\nbf - (\Pibfzp \Dnt\eta)\nbf \big\}
    \llbracket \nabla \xi \rrbracket
=: A_1 + A_2.
\end{split}
\end{equation}
Since~$\Pibfzp\cdot$ preserves polynomials of degree~$\p$,
we have
\begin{equation} \label{preliminary-fact-projectors}
\Dnt \eta - \Pibfzp \Dnt \eta
= \Dnt (u-\un) - \Pibfzp (\Dnt u) + \Pibfzp (\Dnt \un)
= \Dnt u - \Pibfzp \Dnt u.
\end{equation}
Using~\eqref{dg-norm},
\eqref{dg-parameters},
and~\eqref{preliminary-fact-projectors},
we write
\begin{equation} \label{estimate:A1}
\begin{split}
A_1
& \lesssim \frac{\h^{\frac32}}{\p^3}
         \Norm{\nabla\cdot (\Dnt \eta - \Pibfzp \Dnt\eta)}_{0,\Fcaln}
         \Norm{\sigmabold^{\frac12} \llbracket \xi \rrbracket}_{0,\Fcaln}
\lesssim \frac{\h^{\frac32}}{\p^3}
         \Norm{\nabla\cdot (\Dnt u - \Pibfzp \Dnt u)}_{0,\Fcaln}
         \Norm{\xi}_{\dG}.
\end{split}
\end{equation}
Next, we focus on the term~$A_2$ and use similar manipulations:
\begin{equation} \label{estimate:A2}
\begin{split}
A_2
& \lesssim \frac{\h^{\frac12}}{\p}
         \Norm{\Dnt \eta - \Pibfzp \Dnt \eta}_{0,\Fcaln}
         \Norm{\taubold^{\frac12} \llbracket \xi \rrbracket}_{0,\Fcaln}
\lesssim \frac{\h^{\frac12}}{\p}
         \Norm{\Dnt u - \Pibfzp \Dnt u}_{0,\Fcaln}
         \Norm{\xi}_{\dG}.\\
\end{split}
\end{equation}
We insert~\eqref{estimate:A1} and~\eqref{estimate:A2}
in~\eqref{starting-point} and get
\[
\begin{split}
\vert \Bn(u-\un,\xi) \vert
& \lesssim \left[
    \frac{\h^{\frac32}}{\p^3} \Norm{\nabla\cdot (\Dnt u - \Pibfzp \Dnt u)}_{0,\Fcaln}
    + \frac{\h^{\frac12}}{\p} \Norm{\Dnt u - \Pibfzp \Dnt u}_{0,\Fcaln}  \right] \Norm{\xi}_{\dG}.
\end{split}
\]
Thus, we arrive at the following bound on the inconsistency term in~\eqref{error-splitting}:
\begin{equation} \label{T1+T2}
\sup_{\xi \in \Vn} \frac{\vert\Bn(u-\un,\xi)\vert}{\Norm{\xi}_{\dG}}
\lesssim
\frac{\h^{\frac32}}{\p^3} \Norm{\nabla\cdot (\Dnt u - \Pibfzp \Dnt u)}_{0,\Fcaln}
    + \frac{\h^{\frac12}}{\p} \Norm{\Dnt u - \Pibfzp \Dnt u}_{0,\Fcaln}
    =: T_1 + T_2.
\end{equation}
To conclude, we are left to estimate the terms~$T_1$ and~$T_2$.
First, we consider the term~$T_1$
and estimate the $L^2$ norm on each face~$\F$ of the following vector:
\[
\nabla\cdot (\Dnt u - \Pibfzp \Dnt u) =
\begin{bmatrix}
\partialx (\partialx^2 u   - \Pizp \partialx^2 u)
+ \partialy (\partialxy u  - \Pizp \partialxy u)  \\
\partialx (\partial_{yx} u - \Pizp \partial_{yx} u)
+ \partialy (\partialy^2 u - \Pizp \partialy^2 u)
\end{bmatrix}.
\]
Using the triangle inequality, we have to estimate the four terms appearing above.
Since each of them can be handled similarly, we only show the estimates for one of them.
Moreover, we show local estimates, say,
on each face~$\F$ belonging (at least) to the element~$\E$,
as the global counterparts follow immediately.
The triangle inequality implies
\begin{equation} \label{Z1-Z2}
\Norm{\partialx (\partialx^2u - \Pizp\partialx^2u)}_{0,\F}
\le
\Norm{\partialx^3 u - \Pizpmo \partialx^3 u}_{0,\F}
+
\Norm{\Pizpmo \partialx^3 u - \partialx \Pizp \partialx^2 u}_{0,\F}
=: Z_1 + Z_2.
\end{equation}
As for the term~$Z_1$, Lemma~\ref{lemma:L2bulk:L2edge} yields
\begin{equation} \label{Z1}
Z_1 \lesssim \left( \frac{\h}{\p}  \right) ^{s-\frac32} \Norm{\partialx^3 u}_{s-1,\E}
\le \left( \frac{\h}{\p}  \right) ^{s-\frac32} \Norm{u}_{s+2,\E}
\qquad \forall s \ge 0.
\end{equation}
As for the term~$Z_2$, we use the polynomial inverse trace inequality~\cite[eq. (4.6.4)]{Schwab:1997}
(which gives an order suboptimality in~$\p$)
and the triangle inequality:
\begin{equation} \label{Z21-Z22}
\begin{split}
Z_2
&   \lesssim  \frac{\p}{\h^{\frac12}} \Norm{\Pizpmo \partialx^3 u - \partialx \Pizp \partialx^2 u}_{0,\E} \\
&   \le \frac{\p}{\h^{\frac12}}
                \left[  \Norm{\partialx^3 u - \Pizpmo \partialx^3 u}_{0,\E}
                + \Norm{\partialx(\partialx^2 u - \Pizp \partialx^2 u)}_{0,\E} \right]
    = Z_{2,1} + Z_{2,2}.
\end{split}
\end{equation}
Using the approximation properties of the~$L^2$ projector yields
\[
Z_{2,1}
\lesssim \frac{\p}{\h^{\frac12}}  \left(\frac{\h}{\p} \right)^{s-1} \Norm{\partialx^3 u}_{s-1,\E}
\le \frac{\p}{\h^{\frac12}}  \left(\frac{\h}{\p} \right)^{s-1} \Norm{u}_{s+2,\E} .
\]
On the other hand, standard manipulations entail
\[
\begin{split}
Z_{2,2}
& \lesssim \frac{\p}{\h^{\frac12}}
            \left[
                \Norm{\partialx(\partialx^2 u - \Pizyp \partialx^2 u)}_{0,\E}
                + \Norm{\partialx(\Pizyp \partialx^2 u - \Pizyp \Pizxp \partialx^2 u)}_{0,\E}
            \right] \\
& \le \frac{\p}{\h^{\frac12}}
            \left[
                \Norm{\partialx(\partialx^2 u - \Pizyp \partialx^2 u)}_{0,\E}
                + \Norm{\partialx(\partialx^2 u - \Pizxp \partialx^2 u)}_{0,\E}
            \right].
\end{split}
\]
Using the approximation properties of the $L^2$ projection on the first term,
see Lemma~\ref{lemma:L2-H1:projectors},
and~\eqref{H1-L2:estimates-CQ} (which gives half an order suboptimality in~$\p$) on the second term imply
\[
Z_{2,2}
\lesssim
\frac{\p^{\frac32}}{\h^{\frac12}} \left( \frac{\h}{\p} \right)^{s-1} \Norm{\partialx^3 u}_{s-1,\E}
\lesssim
\frac{\p^{\frac32}}{\h^{\frac12}} \left( \frac{\h}{\p} \right)^{s-1} \Norm{u}_{s+2,\E} .
\]
We plug the estimates on~$Z_{2,1}$ and~$Z_{2,2}$ in~\eqref{Z21-Z22} and arrive at
\begin{equation} \label{Z2}
Z_2
\lesssim
\frac{\p^{\frac32}}{\h^{\frac12}} \left( \frac{\h}{\p} \right)^{s-1} \Norm{ u}_{s+2,\E},
\qquad s \ge 0.
\end{equation}
Inserting next~\eqref{Z1} and~\eqref{Z2} in~\eqref{Z1-Z2} gives
\[
\Norm{\partialx (\partialx^2u - \Pizp\partialx^2u)}_{0,\F}
\lesssim \frac{\p^{\frac32}}{\h^{\frac12}} \left( \frac{\h}{\p} \right)^{s-1} \Norm{ u}_{s+2,\E}
=  \frac{\p^{\frac52}}{\h^{\frac32}} \left( \frac{\h}{\p} \right)^{s} \Norm{ u}_{s+2,\E},
\qquad s \ge 0.
\]
In overall, also recalling the factor~$\frac{\h^{\frac32}}{\p^3}$ in front of~$T_1$,
we deduce
\begin{equation} \label{T1}
T_1
\lesssim \p^{-\frac12}
\left(\frac{\h}{\p}\right)^{s}  \Norm{u}_{s+2, \Omega}.
\end{equation}
As for the term~$T_2$ on the right-hand side of~\eqref{T1+T2},
we use Lemma~\ref{lemma:L2bulk:L2edge} and obtain
\begin{equation} \label{T2}
T_2 \lesssim \p^{-\frac12} \left( \frac\h\p \right)^{s}  \Norm{u}_{s+2,\Omega}.
\end{equation}
Inserting~\eqref{T1} and~\eqref{T2} in~\eqref{T1+T2} leads to
\begin{equation} \label{second-term-Strang}
\sup_{\xi \in \Vn} \frac{\vert\Bn(u-\un,\xi)\vert}{\Norm{\xi}_{\dG}}
\lesssim \p^{-\frac12} \left( \frac\h\p \right)^{s}  \Norm{u}_{s+2,\Omega}
\le \left( \frac\h\p \right)^{s}  \Norm{u}_{s+2,\Omega}.
\end{equation}
A combination of~\eqref{starting-point}, \eqref{first-term-Strang}, and~\eqref{second-term-Strang} yields the assertion.
\end{proof}

\begin{remark} \label{remark:nonhomogeneous-BCs}
Theorem~\ref{theorem:convergence} states that
the DG scheme~\eqref{dG} delivers $\h\p$-optimal error estimates
for the biharmonic problem~\eqref{weak-formulation}
with homogeneous boundary conditions.
Following~\cite{Georgoulis-Hall-Melenk:2010},
it is possible to show that $\h\p$-optimal convergence is achieved with polynomial/smooth boundary conditions as well.
On the other hand, singularities arising on the boundary of the domain
yield a suboptimal $\p$-convergence,
which we shall highlight in the numerical experiments;
see Section~\ref{subsection:numerics:singular-boundary} below.
\eremk
\end{remark}

\subsection{Optimal convergence on tensor product-type meshes in weaker norms} \label{subsection:Aubin:Nitsche}
For tensor product-type meshes,
we discuss $\h\p$-optimal error estimates in weaker norms
under the assumption that~$\Omega$ is convex.
Notably, we are interested in the $L^2$ and broken-$H^1$ error estimates.
Since analogous techniques are employed, we focus on $L^2$ error estimates only.

\begin{theorem} \label{theorem:L2-convergence}
Let~$\taun$ be a quasi-uniform tensor product-type mesh of either parallelograms (2D) or parallelepipeds (3D)
as in Section~\ref{subsection:meshes}.
Let~$u$ and~$\un$ be the solutions to~\eqref{weak-formulation} and~\eqref{dG}, respectively.
Then, if~$\Omega$ is convex,
there exists a positive constant~$c$ independent of~$\h$ and~$\p$
such that
\[
\Norm{u-\un}_{0,\Omega}
\le c \left(\frac{\h}{\p} \right)^2 \Norm{u-\un}_{\dG}.
\]
Together with Theorem~\ref{theorem:convergence},
this implies $\h\p$-optimal convergence in the $L^2$ norm
if additional regularity on~$u$ is granted.
\end{theorem}
\begin{proof}
Consider the auxiliary problem: find~$\Phi$ such that
\[
\begin{cases}
\Delta^2 \Phi = u-\un   & \text{in } \Omega,\\
u=0, \quad \nbfOmega \cdot \nabla u =0 & \text{on }\partial\Omega.
\end{cases}
\]
In weak formulation, this problem reads
\begin{equation} \label{auxiliary-problem:L2:AubinNitsche}
\begin{cases}
\text{find } u \in V \text{ such that}\\
B(u,v) = (u-\un,v)_{0,\Omega} \qquad \forall v\in V.
\end{cases}
\end{equation}
Due to the convexity of~$\Omega$, the solution~$\Phi$ to~\eqref{auxiliary-problem:L2:AubinNitsche} belongs to~$H^4(\Omega)$ and the following regularity estimate holds true:
\begin{equation}\label{H4-stability}
\Norm{\Phi}_{4,\Omega} \lesssim \Norm{u-\un}_{0,\Omega};
\end{equation}
see, e.g., \cite{Grisvard:2011} and~\cite[Theorem~$6$, page~$182$]{Kozlov-Mazya:1996}
for the two and three dimensional cases, respectively.

Using~\eqref{auxiliary-problem:L2:AubinNitsche},
splitting the integrals elementwise,
and integrating by parts imply
\[
\Norm{u-\un}_{0,\Omega}^2
= (u-\un, \Delta^2 \Phi)_{0,\Omega}
= \sum_{\E \in \taun} \left(
      -(\nabla(u-\un), \nabla \Delta \Phi)_{0,\E}
      + (u-\un, \nabla \Delta \Phi \cdot \nbfE)_{0,\partial \E}  \right).
\]
Integrating by parts, and
using~\eqref{lifting-operator} and the smoothness of~$\Phi$ give
\[
\begin{split}
\Norm{u-\un}_{0,\Omega}^2
& = \sum_{\E\in\taun} (D^2(u-\un),D^2\Phi)_{0,\E}
    +(\llbracket u-\un \rrbracket, \{ \nabla \Delta \Phi \cdot \nbf  \})_{0,\Fcaln}\\
& \quad + (\llbracket \Phi \rrbracket, \{ \nabla \Delta (u-\un) \cdot \nbf  \})_{0,\Fcaln}
     - (\llbracket \nabla(u-\un) \rrbracket, \{ D^2 \Phi \cdot \nbf  \})_{0,\Fcaln}\\
& \quad - (\llbracket \nabla \Phi \rrbracket, \{ D^2 (u-\un) \cdot \nbf  \})_{0,\Fcaln}
   + ( \sigmabold^{\frac12} \llbracket \un \rrbracket,  \llbracket \vn \rrbracket)_{0,\Fcaln}
    + (\taubold^{\frac12} \llbracket \nabla\un \rrbracket, \llbracket \nabla\vn \rrbracket)_{0,\Fcaln}.
\end{split}
\]
As, e.g., in~\cite[Theorem~$4.5$]{Georgoulis-Lasis:2006},
we write the right-hand side as a combination of~$\Bn(u-\un,\Phi)$ and a reminder~$R$.
Standard computations reveal that
\[
\begin{split}
& R = \left[ (\llbracket u-\un \rrbracket, \{ \nabla\Delta \Phi\cdot \nbf  \})_{0,\Fcaln}
    - (\llbracket \nabla(u-\un) \rrbracket, \{ D^2 \Phi \cdot \nbf  \})_{0,\Fcaln}
    - (\Lcal(u-\un),\Pizp \Dnt \Phi )_{0,\Omega}  \right]\\
&  \!+\!
    \left[ (\llbracket \Phi \rrbracket, \{ \nabla\Delta (u-\un) \cdot \nbf  \})_{0,\Fcaln}
    \!-\! (\llbracket \nabla \Phi \rrbracket, \{ D^2 (u-\un) \cdot \nbf  \})_{0,\Fcaln}
    \!-\! (\Lcal \Phi, \Pizp(\Dnt (u-\un)))_{0,\Omega}  \right]\\
& = A + B.
\end{split}
\]
Due to the regularity of~$\Phi$, we have~$B=0$.
We deduce
\begin{equation} \label{B-R}
\Norm{u-\un}_{0,\Omega}^2
=\Bn(u-\un,\Phi) + R,
\end{equation}
where
\[
\begin{split}
R
& = (\llbracket u-\un \rrbracket,
\{ \nbf \cdot \nabla (D^2 \Phi-\Pizp D^2 \Phi ) \})_{0,\Fcaln}
    - (\llbracket \nabla(u-\un) \rrbracket,
    \{ \nbf \cdot (D^2 \Phi - \Pizp D^2 \Phi ) \})_{0,\Fcaln}.
\end{split}
\]
Thanks to the smoothness of~$\Pcalp \Phi$,
where~$\Pcalp$ is as in Theorem~\ref{theorem:approximation-tensor-product-meshes},
we can write
\[
\Bn(u-\un,\Phi)
= \Bn(u-\un,\Phi-\Pcalp\Phi).
\]
In fact, $\Lcal(\Pcalp \Phi)=0$.
Using the continuity property~\eqref{continuity&coercivity:bf},
the smoothness and the homogeneous essential boundary conditions of~$\Phi-\Pcalp\Phi$,
the approximation properties of~$\Pcalp$ in~\eqref{hp-H2-approximation},
and the stability property~\eqref{H4-stability},
we deduce
\begin{equation} \label{dual-estimate:B}
\begin{split}
\Bn(u-\un,\Phi)
&   \lesssim \Norm{u-\un}_{\dG}
        \Norm{\Dnt (\Phi-\Pcalp\Phi)}_{0,\Omega}
    \lesssim \left(\frac{\h}{\p}\right)^2
        \Norm{u-\un}_{\dG} \Norm{\Phi}_{4,\Omega}\\
&   \lesssim \left(\frac{\h}{\p}\right)^2
        \Norm{u-\un}_{\dG} \Norm{u-\un}_{0,\Omega}.
\end{split}
\end{equation}
We are left with estimating the remainder~$R$ on the right-hand side of~\eqref{B-R}.
To this aim, we proceed as in the proof of Theorem~\ref{theorem:convergence}.
More precisely, we apply the estimates~\eqref{T1} and~\eqref{T2} with~$s=2$
and use the stability bound~\eqref{H4-stability} to deduce
\begin{equation} \label{dual-estimate:R}
R
\lesssim \Norm{u-\un}_{\dG}
            \left(\frac{\h}{\p}\right)^2 \Norm{\Phi}_{4,\Omega}
\lesssim \left(\frac{\h}{\p}\right)^2
            \Norm{u-\un}_{\dG}  \Norm{u-\un}_{0,\Omega}.
\end{equation}
The assertion follows combining~\eqref{dual-estimate:B} and~\eqref{dual-estimate:R} in~\eqref{B-R}.
\end{proof}

\begin{remark} \label{remark:no-convexity}
As customary, the convexity assumption in Theorem~\ref{theorem:L2-convergence} can be relaxed by demanding that
the domain~$\Omega$ implies $H^4$ regularity on the solution of the auxiliary problem~\eqref{auxiliary-problem:L2:AubinNitsche}.
\eremk
\end{remark}

\subsection{Optimal convergence on triangular meshes}
\label{subsection:optimal-convergence-triangle}
In this section, we provide the guidelines to extend
Theorems~\ref{theorem:convergence} and~\ref{theorem:L2-convergence}
to the case of triangular meshes.
As in the analysis above,
$\h\p$-optimal convergence estimates follow from
the existence of an $H^2$ conforming polynomial interpolant
over a given triangular mesh with $\h\p$-optimal approximation properties.
We undertake a different approach compared to the case of tensor product-type meshes
as the arguments therein employed do not apply on simplicial meshes.

In~\cite{Suri:1990}, Suri constructed an interpolant~$\Ibs$ of $H^2$ functions
into the space of global $H^2$ piecewise  polynomials on triangular meshes.
Such an interpolant satisfies the following approximation properties:
there exists a positive constant~$c$ only depending on~$\taun$ such that
\begin{equation} \label{BS-H2}
\SemiNorm{u-\Ibs u}_{\ell,\Omega}
\le c p^{-(k-\ell)}  \SemiNorm{u}_{k,\Omega}
\qquad \qquad \ell =0,1,2.
\end{equation}
The interpolant~$\Ibs$ is constructed generalizing to the $H^2$ case
the techniques for the $H^1$ case
developed by Babu\v ska and Suri in~\cite{Babuska-Suri:1987}.
First, a polynomial with $p$-optimal approximation properties is constructed on each triangle.
Then, global $H^2$ regularity is recovered
using an $H^2$ stable lifting operator,
which preserves essential polynomial conditions on the boundary of a triangle.
So, $p$-optimal error bounds of the IPDG on triangular meshes for the biharmonic problem
follow using~\eqref{BS-H2} in the analysis of
Section~\ref{subsection:optimal-tensor-product}.

We report the convergence results for triangular meshes.
\begin{theorem} \label{theorem:convergence-triangular-meshes}
Let~$\taun$ be a quasi-uniform triangular mesh as in Section~\ref{subsection:meshes}.
Let~$u$ and~$\un$ be the solutions to~\eqref{weak-formulation} and~\eqref{dG}, respectively.
Assume that~$u\in H^{s+2}(\Omega)$, $s:=\min(k,\p-1)$, $\p \ge 2$.
Then, there exists a positive constant~$c$
independent of~$\h$ and~$\p$
such that
\[
\Norm{u-\un}_{\dG}
\le c \left(\frac{\h}{\p}\right)^{s} \Norm{u}_{s+2,\Omega},
\qquad\qquad
\Norm{u-\un}_{0,\Omega}
\le c \left(\frac{\h}{\p} \right)^2 \Norm{u-\un}_{\dG}.
\]
In other words, for homogeneous essential boundary conditions,
method~\eqref{dG} is $\h\p$-optimal in the energy and~$L^2$ norms
also on triangular meshes.
\end{theorem}

\begin{remark} \label{remark:H^2-BS}
The $p$-optimal interpolant in \eqref{BS-H2} is derived based on the trigonometric polynomials approximation theory.
The error bound is $p$-optimal with respect to the Sobolev index~$k$
but contains unknown constants depending on $k$;
this is due to the use of the Stein's extension operator for Lipschitz domains~\cite{Stein:1970}.
The corresponding bound cannot be used to derive exponential convergence for the $hp$-version of Galerkin methods
but is enough to study $p$-optimal algebraic convergence.
So, one of the reasons why we developed the theory of Section~\ref{section:hp-polynomial}
is that the $\h\p$-optimal error estimates are explicit in the Sobolev index~$k$ on tensor product-type meshes.
\eremk
\end{remark}

\begin{remark}\label{3D}
To the best of our knowledge, so far,
the Suri-type $p$-optimal interpolation operator in~\eqref{BS-H2}
can be constructed in two dimensions only.
There are two main reasons for this fact:
(\emph{i}) an explicit construction of general order $H^2$ conforming piecewise polynomial spaces on tetrahedral meshes
is still missing;
(\emph{ii}) $H^2$ stable lifting operators for tetrahedra are also unknown.
Notwithstanding, it is possible to use $H^1$-conforming $\h\p$-optimal polynomial interpolation operator on tetrahedral meshes.
The resulting error bound for the IPDG is then suboptimal only by half an order of~$p$,
thus improving the expected one and a half an order suboptimality.
\eremk
\end{remark}

\begin{remark}\label{remark:H^2-polynomial-lifting}
There exist at least other two different constructions for $H^2$ stable lifting operators.
In~\cite{Lederer-Schoeberl:2018}, an $H^2$ stable lifting operator has been derived to prove the $p$-robustness of the inf-sup constant
in H(div)-FEM for the Stokes problem in two dimensions.
In~\cite[Theorem~$2.2$]{Ainsworth-Parker:2020},
the existence of an $H^2$ stable lifting operator was shown,
which preserves essential polynomial conditions on the boundary of a triangle
to design preconditioners for high-order $H^2$ FEM on triangular meshes~\cite{Ainsworth-Parker:2021}.
\eremk
\end{remark}

\section{Convergence for other IPDG formulations of the biharmonic problem} \label{section:variants}

In this section, we discuss other IPDG formulations that can be proven to be $\h\p$-optimal
using similar techniques to those of Section~\ref{section:convergence}.
\medskip

\subsection{$\mathcal{C}^0$-IPDG for the biharmonic problem}
In this section, we consider $\mathcal C^0$-IPDG for the biharmonic problem;
see~\cite{Engel-Garikipati-Hughes-Larson-Mazzei-Taylor:2002,Brenner-Sung:2005}.
Given the space~$\Vn$ of piecewise continuous polynomials over a mesh~$\taun$ as in Section~\ref{subsection:meshes},
we introduce the lifting operator~$\Lcal: \Vn \oplus V \to [\Vn]^{d\times d}$ as
\[
\int_{\Omega} \Lcal(\un) : \vbfn
:=  - \int_{\Fcaln} (\{ (\vbfn) \nbf \} \cdot \nbf) \; (\llbracket \nabla\un \rrbracket \cdot  \nbf)
\qquad  \forall \vbfn \in [\Vn]^{d\times d}
\]
and the discrete bilinear form
\begin{equation}\label{C0IP-DG}
\begin{split}
\Bn(\un,\vn)
& := \int_\Omega \Dnt \un : \Dnt \vn
  + \int_{\Omega} \left( \Lcal(\un):\Dnt\vn + \Lcal(\vn) :\Dnt \un    \right)\\
& \quad  + \int_{\Fcaln}
        \taubold (\llbracket \nabla\un \rrbracket \cdot  \nbf)  (\llbracket \nabla\vn \rrbracket \cdot \nbf)
\qquad \forall \un,\vn\in\Vn.
\end{split}
\end{equation}
With an abuse of notation,
we denote the norm induced by the bilinear form in~\eqref{C0IP-DG} by~$\Norm{\cdot}_{\dG}$.

The bilinear form~$\Bn(\cdot,\cdot)$ is equal to that in~\eqref{bilinear-form}
without the scalar jump terms.
Moreover, the jump of the gradients appear only along the normal direction to the faces.
Thus, $p$-optimal error estimates follow as
(and are in fact easier to prove than) in Section~\ref{section:convergence}.

For completeness, we report the convergence result.
\begin{theorem} \label{theorem:convergence-C0-IPDG}
Let~$\taun$ be either a quasi-uniform triangular mesh,
or a quasi-uniform tensor product-type mesh of either parallelograms (2D) or parallelepipeds (3D)
as in Section~\ref{subsection:meshes}.
Let~$u$ and~$\un$ be the solutions to~\eqref{weak-formulation} and the corresponding $\mathcal C^0$-IPDG, respectively.
Assume that~$u\in H^{s+2}(\Omega)$, $s:=\min(k,\p-1)$, $\p \ge 2$.
Then, there exists a positive constant~$c$
independent of~$\h$, and~$\p$
such that
\[
\Norm{u-\un}_{\dG}
\le c \left(\frac{\h}{\p}\right)^{s} \Norm{u}_{s+2,\Omega},
\qquad\qquad
\Norm{u-\un}_{0,\Omega}
\le c \left(\frac{\h}{\p} \right)^2 \Norm{u-\un}_{\dG}.
\]
\end{theorem}

\subsection{Divergence-conforming DG methods for the stream formulation of the Stokes problem in two dimensions}
In this section,
we discuss how to extend the optimal results of Section~\ref{section:convergence}
to divergence-conforming DG methods for
the stream formulation of the Stokes problem \emph{in two dimensions;}
see~\cite{Kanschat-Sharma:2014,Wang-Wang-Ye:2009}.

We start with the standard formulation of the Stokes problem in two dimensions:
given a polygonal domain~$\Omega \subset \Rbb^2$
and~$\fbf \in [L^2(\Omega)]^2$,
we seek~$\ubftilde \in [H^1_0(\Omega)]^2$
and~$\rho \in L^2_0(\Omega)$ that solve weakly
\begin{equation} \label{Stokes:standard-formulation}
\begin{cases}
-\Delta \ubftilde + \nabla \rho = \fbf & \text{in } \Omega, \\
\nabla\cdot \ubftilde =0            & \text{in } \Omega, \\
\ubftilde = \zerobf                 & \text{on } \partial \Omega.
\end{cases}
\end{equation}
The velocity~$\ubftilde$ is divergence free.
So, there exists~$\psi \in H^2_0(\Omega)$ such that
$\ubftilde = \nabla \times \psi$.

Introduce the Stokes kernel
\[
\Zbf:=
\left\{ 
    \vbf \in [H^1_0(\Omega)]^2 \ \middle| \
            \div \vbf = 0 \right\}.
\]
Testing the first equation in~\eqref{Stokes:standard-formulation}
with~$\vbftilde \in \Zbf$ gives
\[
(\nabla \ubftilde,\nabla\vbftilde)_{0,\Omega}
= (\nabla \ubftilde,\nabla\vbftilde)_{0,\Omega}
+ (\nabla \rho, \vbftilde)_{0,\Omega}
= (\fbf, \vbftilde)_{0,\Omega}.
\]
Since~$\vbftilde \in \Zbf$,
there exists~$\phi \in H^2_0(\Omega)$ such that
$\vbftilde = \nabla \times \phi$.
We deduce the stream formulation of the Stokes problem in~\eqref{Stokes:standard-formulation}:
\begin{equation} \label{weak-formulation-stream}
(\nabla \ (\nabla \times \psi), \nabla \ (\nabla \times \phi))_{0,\Omega}
= (\fbf, \nabla \times \phi)_{0,\Omega} \qquad \forall \phi \in H^2_0(\Omega).
\end{equation}
This is a weak formulation of the biharmonic problem~\eqref{strong-formulation}
which is alternative to that in~\eqref{weak-formulation}.
We can rewrite the right-hand side of~\eqref{weak-formulation-stream} integrating by parts:
\[
(\fbf, \nabla \times \phi)_{0,\Omega} = (\nabla \times \fbf, \phi)_{0,\Omega}.
\]
The strong formulation corresponding to~\eqref{weak-formulation-stream} reads as follow:
we seek~$\psi \in H^2_0(\Omega)$ satisfying
\[
\Delta^2 \psi = \nabla \times \fbf \text{ in }\Omega.
\]
We derive the $\mathcal{C}^0$-IPDG formulation of this problem
using standard dG techniques;
see, e.g., \cite{Kanschat-Sharma:2014}.
On piecewise polynomial spaces,
we define the discrete bilinear form
\[
\begin{split}
\Bn(\psi_n,\phi_n)
& := \int_{\taun} (\nabla \times \nabla \psi_n) : (\nabla\times \nabla \phi_n)
   - \int_{\Fcaln} \{ \nabla\times\nabla\psi_n \} \cdot \llbracket \nabla\phi_n \times \nbf \rrbracket\\
& \quad    - \int_{\Fcaln} \{ \nabla\times\nabla\phi_n \} \cdot \llbracket \nabla\psi_n \times \nbf \rrbracket
   + \int_{\Fcaln} \taubold \llbracket \nabla\psi_n \times \nbf \rrbracket \cdot \llbracket \nabla \phi_n \times \nbf \rrbracket
\qquad \forall \psi_n, \phi_n \in\Vn.
\end{split}
\]
With an abuse of notation,
we denote the norm induced by the bilinear form in~\eqref{C0IP-DG} by~$\Norm{\cdot}_{\dG}$.

Following the same steps as in Section~\ref{section:convergence},
$\h\p$-optimal error estimates can be proved.
In fact, the problem is 
the bilinear form appearing in $\mathcal C^0$-IPDG \eqref{C0IP-DG} with rotated components.

For completeness, we report the convergence result.
\begin{theorem} \label{theorem:convergence-Stokes}
Let~$\taun$ be either a quasi-uniform triangular mesh,
or a quasi-uniform tensor product-type mesh of either parallelograms (2D) or parallelepipeds (3D)
as in Section~\ref{subsection:meshes}.
Let~$\psi$ and~$\psi_n$ be the solutions to~\eqref{weak-formulation-stream} and the corresponding $\mathcal C^0$-IPDG, respectively.
Assume that~$\psi \in H^{s+2}(\Omega)$, $s:=\min(k,\p-1)$, $\p \ge 2$.
Then, there exists a positive constant~$c$
independent of~$\h$, and~$\p$
such that
\[
\Norm{\psi-\psi_n}_{\dG}
\le c \left(\frac{\h}{\p}\right)^{s} \Norm{\psi}_{s+2,\Omega},
\qquad\qquad
\Norm{\psi-\psi_n}_{0,\Omega}
\le c \left(\frac{\h}{\p} \right)^2 \Norm{\psi-\psi_n}_{\dG}.
\]
\end{theorem}

\section{Numerical experiments} \label{section:numerical-experiments}
In this section, we validate the theoretical results of Section~\ref{section:convergence}.
More precisely,
in Section~\ref{subsection:numerics:inside-singularity},
we show $\p$-optimal convergence for a test case with internal point singularity
lying inside one element;
in Section~\ref{subsection:numerics:skeleton-singularity},
we test the convergence for a test case with internal point singularity lying on the skeleton
and observe a $\p$-optimal doubling of the convergence rate~\cite{Babuska-Suri:1987};
in Section~\ref{subsection:numerics:singular-boundary},
we consider a test case that is singular on the boundary of the domain,
i.e., a case not covered in the foregoing analysis,
and show a $p$-suboptimal rate of convergence.

\begin{figure}[ht]
\centering
 \includegraphics[height=1.5in,width=1.5in]{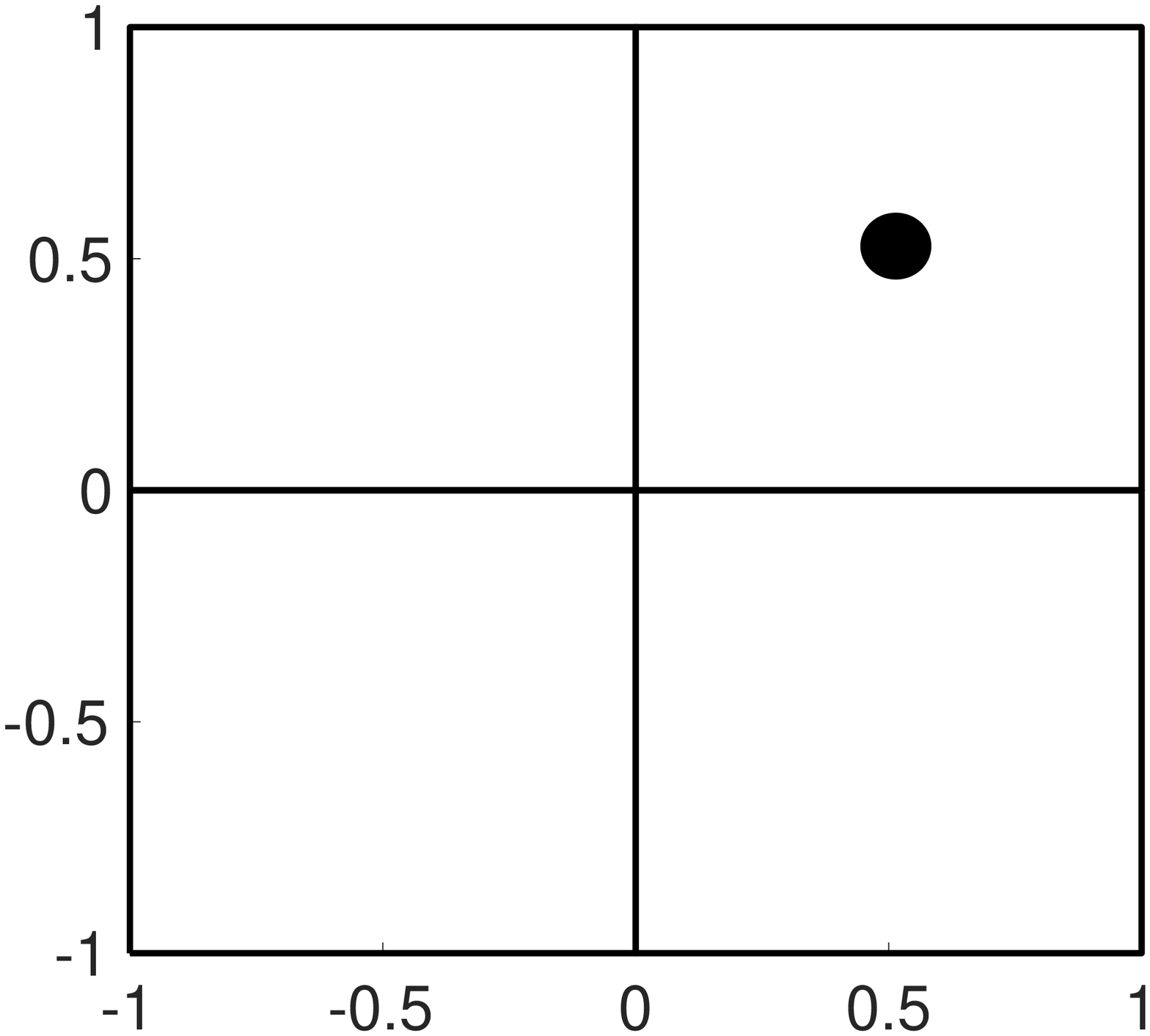}
 \qquad\qquad
 \includegraphics[height=1.5in,width=1.5in]{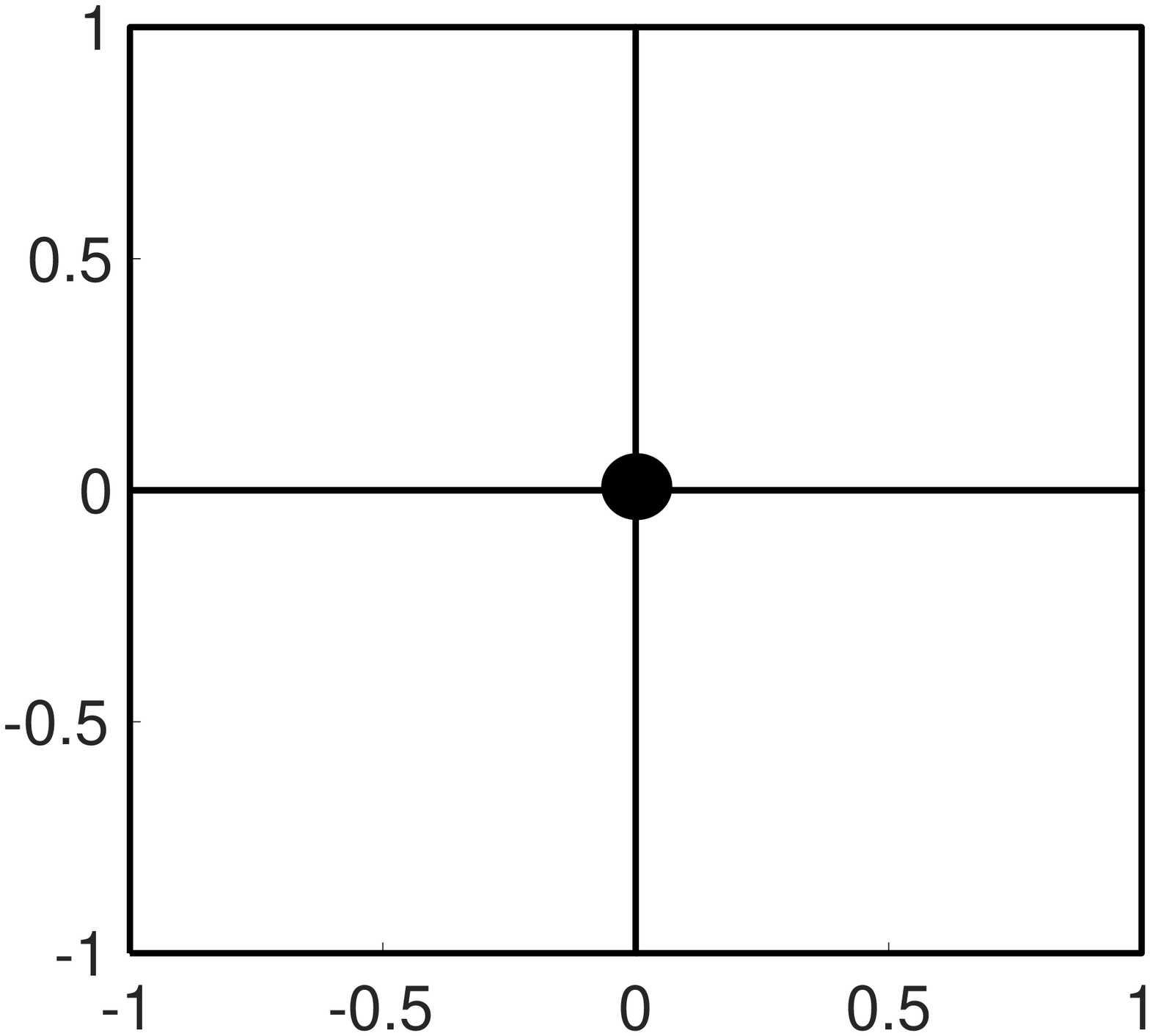}
 \qquad\qquad
 \includegraphics[height=1.5in,width=1.5in]{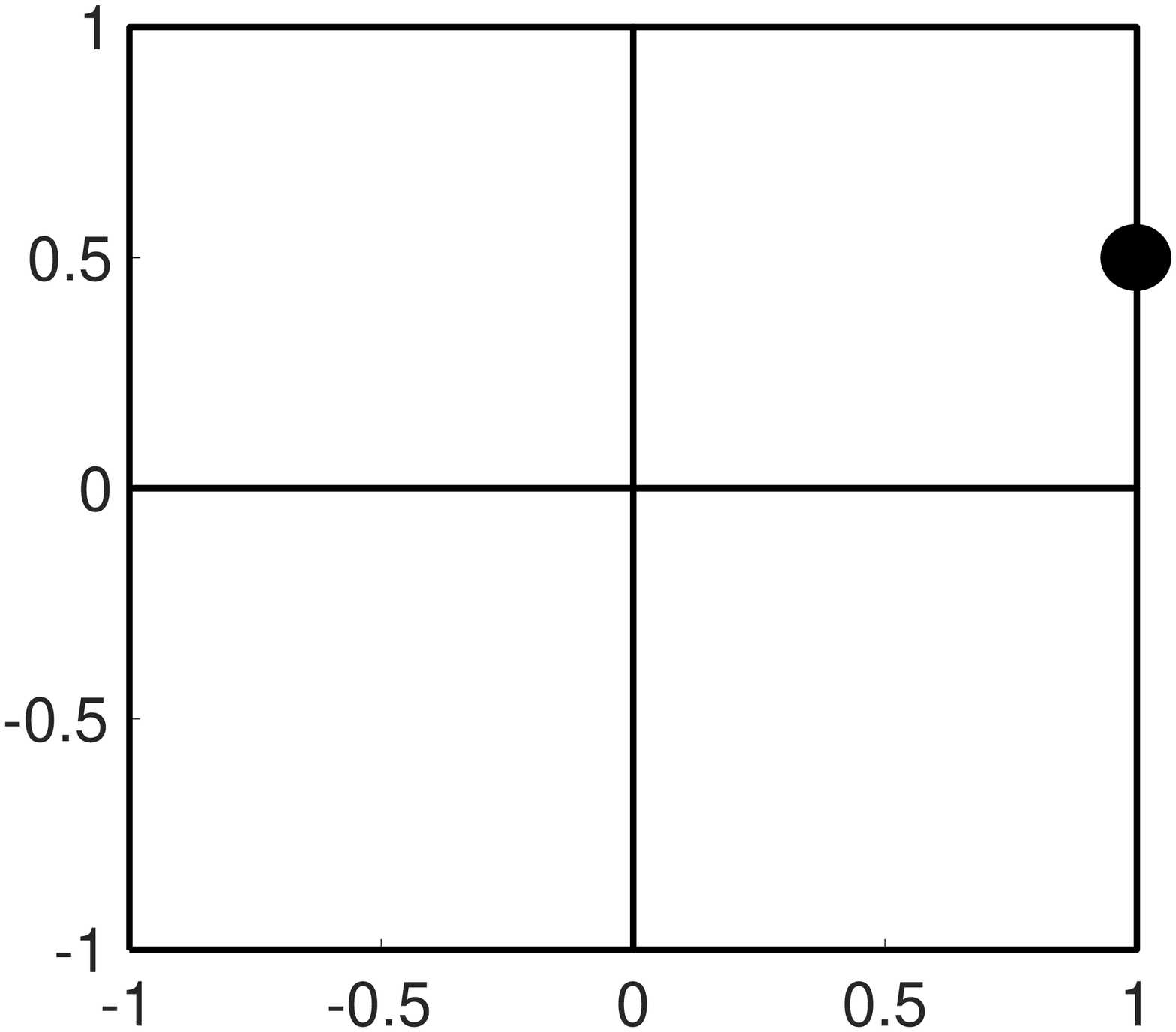}
\caption{The three different point singularities we shall consider.
\emph{Left panel}: point singularity inside one element.
\emph{Central panel}: point singularity at an internal mesh vertex.
\emph{Right panel}: point singularity on the boundary.}
 \label{fig:picture}
\end{figure}

Below, we consider the square domain~$\Omega:=(-1,1)^2$
and a given uniform Cartesian mesh of four elements
in Sections~\ref{subsection:numerics:inside-singularity},
\ref{subsection:numerics:skeleton-singularity},
and~\ref{subsection:numerics:singular-boundary},
and the L-shape domain
$\Omega := (-1,1)^2 \setminus [0,1) \times (-1,0]$
and a given uniform triangular mesh of 6 elements
in Section~\ref{subsection:numerics:L-shape}.

\subsection{Point singularity inside one element} \label{subsection:numerics:inside-singularity}
We consider the exact solution
\[
u_1(x,y):= \left((x-0.5)^2+(y-0.5)^2 \right)^{\frac{3}{2}}(1-x^2)(1-y^2).
\]
The function~$u_1$ has a point singularity at~$(0.5,0.5)$
that lies in the interior of an element of the mesh;
see the left panel in Figure~\ref{fig:picture}.
It belongs to~$H^{4-\varepsilon}(\Omega)$ for any arbitrarily small~$\varepsilon>0$.
We list the convergence rates under $p$-refinement in Table~\ref{ex1:convergence-internal}.
The results confirm the optimal error decay $\mathcal{O}(p^{-2})$ for the DG-norm error.

\begin{table}[!htb]
	\centering
		\begin{tabular}{|c|c|c||c|c|c|}
			\hline
			$p$  & $\Norm{u-u_n}_{\dG}$ & Convergence rate &  $p$  & $\Norm{u-u_n}_{\dG}$ & Convergence rate  \\
			\hline
			$2$ & 1.46E+01  &   & $3$ &  1.18E+01  &      \\
			\hline
			$4$ & 5.07E+00  & 1.52  & $5$ & 9.67E-01  &   4.90  \\
			\hline
			$6$ & 1.42E-01  & 8.83  & $7$ & 4.19E-02  &   9.33   \\
			\hline
			$8$ & 2.38E-02  &  6.18  & $9$ & 1.62E-02  &    3.79   \\
            \hline
			$10$ & 1.11E-02  & 3.41  & $11$ & 9.29E-03 &   2.77   \\
			\hline
            $12$ & 6.76E-03  & 2.73  & $13$ & 6.12E-03  &   2.49   \\
			\hline
			$14$ &  4.67E-03  & 2.40  & $15$ & 4.39E-03&   2.33   \\
			\hline
			$16$ & 3.48E-03  & 2.21  & $17$ & 3.33E-03  &   2.20   \\
			\hline
			$18$ & 2.73E-03  & 2.05  & $19$ & 3.33E-03  &  2.09   \\
            \hline
			$20$ & 2.23E-03  & 1.93  & $21$ & 2.16E-03  & 1.99  \\
            \hline
			$22$ & 1.84E-03  & 2.03  & $23$ & 1.80E-03  &   2.01   \\
            \hline
			$24$ & 1.55E-03  & 1.96  & $25$ & 1.53E-03  &   1.92   \\
			\hline
		\end{tabular}
	\caption{$\p$-convergence for the exact solution~$u_1$.}
	\label{ex1:convergence-internal}
\end{table}

\subsection{Point singularity at an internal mesh vertex} \label{subsection:numerics:skeleton-singularity}
We consider the exact solution
\[
u_2(x,y):= \left(x^2+y^2 \right)^{\frac{3}{2}}(1-x^2)(1-y^2).
\]
The function~$u_2$ has a point singularity at~$(0,0)$,
which lies at an internal vertex of the mesh;
see the central panel in Figure~\ref{fig:picture}.
It belongs to $H^{4-\varepsilon}(\Omega)$ for any arbitrarily small~$\varepsilon>0$.
We list the convergence rates under $p$-refinement
in Table~\ref{ex2:convergence-internal-vertex}
and observe a doubling rate of the error decay $\mathcal{O}(p^{-4})$ for the DG-norm error.
This phenomenon can be proven based on the optimal results in Section~\ref{section:convergence}
and the conforming finite element arguments
in~\cite{Babuska-Suri:1987,Suri:1990}.

\begin{table}[!htb]
	\centering
		\begin{tabular}{|c|c|c||c|c|c|}
			\hline
			$p$  & $\Norm{u-u_n}_{\dG}$ & Convergence rate &  $p$  & $\Norm{u-u_n}_{\dG}$ & Convergence rate  \\
			\hline
			$2$ & 5.16E+00  &   & $3$ & 4.39E+00  &      \\
			\hline
			$4$ & 2.98E+00  & 0.79  & $5$ & 9.61E-01  &   2.97  \\
			\hline
			$6$ & 8.64E-02  & 8.74  & $7$ & 1.08E-02  &   13.33   \\
			\hline
			$8$ & 6.22E-03  &  9.14  & $9$ & 3.84E-03  &   4.14   \\
            \hline
			$10$ & 2.47E-03  & 4.14  & $11$ & 1.67E-03 &   4.13   \\
			\hline
            $12$ & 1.18E-03  & 4.07  & $13$ & 8.52E-04  &   4.04   \\
			\hline
			$14$ &  6.34E-04  & 3.99  & $15$ & 4.84E-04  &   3.96   \\
			\hline
			$16$ & 3.76E-04  & 3.92  & $17$ & 2.90E-04  &   4.09   \\
			\hline
			$18$ & 2.31E-04  & 4.12  & $19$ & 1.87E-04  &   3.92   \\
            \hline
			$20$ & 1.53E-04  & 3.90  & $21$ & 1.27E-04  &   3.86   \\
            \hline
			$22$ & 1.05E-04  & 3.95  & $23$ & 8.82E-05  &   4.01   \\
            \hline
			$24$ & 7.44E-05  & 3.99  & $25$ & 6.28E-05  &   4.06   \\
			\hline
		\end{tabular}
	\caption{$\p$-convergence for the exact solution~$u_2$.}
	\label{ex2:convergence-internal-vertex}
\end{table}

\subsection{Point singularity on the boundary for the square domain} \label{subsection:numerics:singular-boundary}
We consider the exact solution
\[
u_3(x,y):= \left((x-1)^2+(y-0.5)^2 \right)^{\frac{3}{2}}.
\]
The function~$u_3$ has a point singularity at~$(1,0.5)$,
which lies on the boundary of~$\Omega$;
see the right panel in Figure~\ref{fig:picture}).
In particular, it belongs to $H^{4-\varepsilon}(\Omega)$ for any arbitrarily small~$\varepsilon>0$.
We list the convergence rates under $p$-refinement in Table~\ref{ex3:convergence rate for boundary of the domain ihomogeneous}
and observe a suboptimal by~$1.5$ order rate $\mathcal{O}(p^{-\frac{1}{2}})$ for the DG-norm error.
This suboptimality is the counterpart
of what is observed in~\cite{Georgoulis-Hall-Melenk:2010}
for elliptic partial differential equations of second order.

\begin{table}[!htb]
	\centering
		\begin{tabular}{|c|c|c||c|c|c|}
			\hline
			$p$  & $\Norm{u-u_n}_{\dG}$ & Convergence rate &  $p$  & $\Norm{u-u_n}_{\dG}$ & Convergence rate  \\
			\hline
			$2$ & 8.61E+00  &   & $3$ &  2.57E+00  &      \\
			\hline
			$4$ & 6.39E-01  & 3.75  & $5$ & 1.01E+00  &   1.83  \\
			\hline
			$6$ & 4.90E-01  & 0.65  & $7$ & 7.48E-01  &   0.89   \\
			\hline
			$8$ & 4.33E-01  &  0.44  & $9$ & 6.37E-01  &    0.64   \\
            \hline
			$10$ & 3.91E-01  & 0.46  & $11$ & 5.71E-01 &   0.55   \\
			\hline
            $12$ & 3.56E-01  & 0.52  & $13$ & 5.24E-01  &   0.51   \\
			\hline
			$14$ &  3.31E-01  & 0.46  & $15$ & 4.88E-01&   0.50   \\
			\hline
			$16$ & 3.13E-01  & 0.42  & $17$ & 4.59E-01  &   0.49   \\
			\hline
			$18$ & 2.97E-01  & 0.47  & $19$ & 4.36E-01  &  0.48   \\
            \hline
			$20$ & 2.82E-01  & 0.48  & $21$ & 4.15E-01  & 0.48  \\
            \hline
			$22$ & 2.69E-01  & 0.49  & $23$ & 3.97E-01  &   0.49   \\
            \hline
			$24$ & 2.58E-01  & 0.49  & $25$ & 3.81E-01  &   0.49  \\
			\hline
		\end{tabular}
	\caption{$\p$-convergence for the exact solution~$u_3$.}
	\label{ex3:convergence rate for boundary of the domain ihomogeneous}
\end{table}


\subsection{Point singularity on the boundary for the L-shape domain} \label{subsection:numerics:L-shape}
On the L-shape domain,
we consider the exact solution
\[
u_4(x,y)=
u_4(r,\theta):=
r^{\frac53} \sin \left( \frac53 \theta \right),
\]
where~$(r,\theta)$ are the polar coordinates at the re-entrant corner $(0,0)$.
We consider the triangular mesh in Figure~\ref{figure:L-shape}
\begin{figure}[ht]
\centering
 \includegraphics[height=1.5in,width=1.5in]{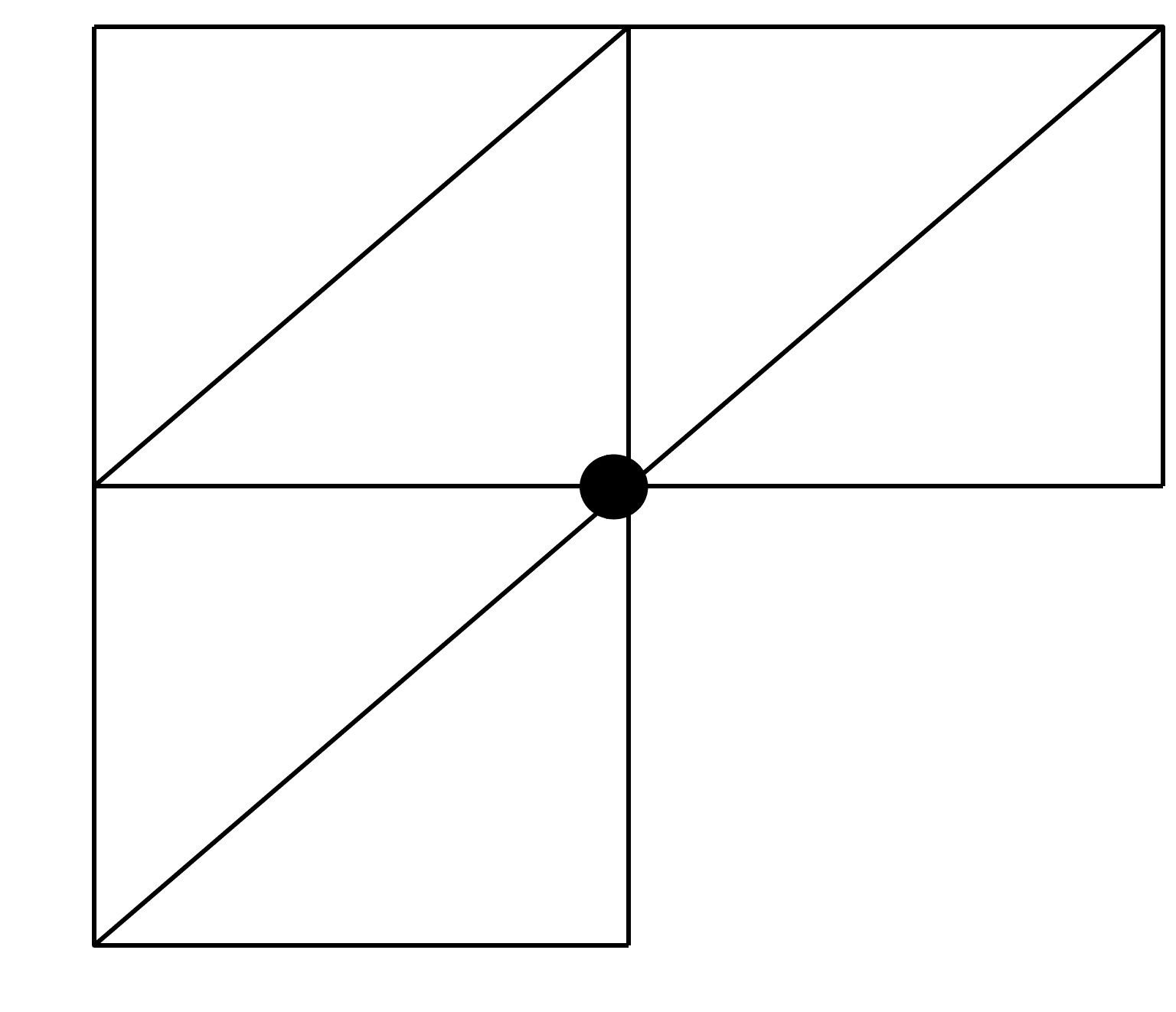}
 \caption{Uniform triangular mesh on the L-shape domain.}
 \label{figure:L-shape}
\end{figure}

The function~$u_4$ has a point singularity at~$(0,0)$
and belongs to~$H^{\frac83-\varepsilon}(\Omega)$ for any arbitrarily small~$\varepsilon>0$;
for this reason, optimal convergence estimates are not guaranteed by the foregoing results, as extra regularity on the solution is required.
Despite this,
we exhibit the usual doubling rate of convergence,
namely the decay rate for the DG-norm error is $\mathcal{O}(p^{\frac{4}{3}})$
as similarly observed for the square domain in Section~\ref{subsection:numerics:singular-boundary};
see Table~\ref{ex4:convergence-L-shape}.

\begin{table}[!htb]
	\centering
		\begin{tabular}{|c|c|c||c|c|c|}
			\hline
			$p$  & $\Norm{u-u_n}_{\dG}$ & Convergence rate &  $p$  & $\Norm{u-u_n}_{\dG}$ & Convergence rate  \\
			\hline
			$2$ & 1.80E+00  &   & $3$ &  9.38E-01  &      \\
			\hline
			$4$ & 6.93E-01  &  1.82  & $5$ & 5.50E-01  &    1.35  \\
			\hline
			$6$ & 4.53E-01  & 1.27  & $7$ & 3.84E-01  &   1.26   \\
			\hline
			$8$ & 3.32E-01  &  1.25  & $9$ & 2.91E-01  &    1.25   \\
            \hline
			$10$ & 2.58E-01 & 1.26  & $11$ & 2.32E-01&   1.26   \\
			\hline
            $12$ &  2.09E-01  & 1.27  & $13$ & 1.91E-01  &   1.27   \\
			\hline
			$14$ &  1.75E-01 & 1.27  & $15$ & 1.62E-01&   1.28   \\
			\hline
			$16$ & 1.50E-01  & 1.28  & $17$ &1.39E-01&   1.29   \\
			\hline
			$18$ & 1.30E-01  & 1.29  & $19$ & 1.22E-01 &  1.28   \\
            \hline
			$20$ & 1.15E-01  & 1.29  & $21$ & 1.08E-01  & 1.30  \\
            \hline
			$22$ & 1.02E-01  & 1.31  & $23$ & 9.60E-02  &   1.31   \\
            \hline
			$24$ & 9.22E-02   & 1.31  & $25$ & 8.77E-02   &   1.32   \\
			\hline
		\end{tabular}
	\caption{$\p$-convergence for the exact solution~$u_1$.}
	\label{ex4:convergence-L-shape}
\end{table}

\section{Conclusions} \label{section:conclusions}
We proved that the IPDG for the biharmonic problem
is $\h\p$-optimal if essential boundary conditions are homogeneous,
and triangular and tensor product-type meshes (in two and three dimensions) are employed.
The analysis hinges upon proving the existence
of a global $H^2$ piecewise polynomial over the given meshes
with $\h\p$-optimal approximation properties.
The results discussed for the biharmonic problem
were extended to other methods,
e.g., $\mathcal C^0$-IPDG, and problems,
e.g., the stream formulation of the Stokes problem.
The theoretical results were validated on several test cases,
which also showed $\p$-suboptimality for solutions with singular essential boundary conditions.

\paragraph*{Competing Interests}
The authors have no relevant financial or non-financial interests to disclose.

\paragraph*{Data Availability.}
The datasets generated during and/or analysed during the current study are available on request.

{\footnotesize \bibliography{bibliogr}} \bibliographystyle{plain}

\appendix

\section{Proof of Lemma~\ref{lemma:H2-1D}} \label{appendix:lemma:H2-1D}

For the sake of presentation, throughout the proof, we shall write~$\up$ instead of~$\Pcalpo u$.

We set~$\up''$ as the Legendre expansion of~$u''$ truncated at order~$\p-2$.
In other words, expanding~$u''$ as a series of Legendre polynomials~$L_j(\xi)$, $j \in \Nbb_0$,
\begin{equation} \label{expanding:u-second}
u''(\xi) = \sum_{i=0}^{+\infty} b_j L_j(\xi),
\end{equation}
we define
\begin{equation} \label{writing:u:1}
\up''(\xi) := \sum_{i=0}^{\p-2} b_j L_j(\xi).
\end{equation}
Standard properties of Legendre polynomials,
see, e.g., \cite[Appendix~C]{Schwab:1997}, imply
\begin{equation} \label{equation:Legendre-properties}
    b_j = \frac{2j+1}{2} \int_{-1}^1 u''(\xi) L_j(\xi) d\xi.
\end{equation}
For~$k\ge0$, we recall~\cite[Lemma~$3.10$]{Schwab:1997} that
\begin{equation} \label{Vk-norm-identity}
\sum_{i=k}^{+\infty} \frac{2}{2i+1} \frac{(i+k)!}{(i-k)!} \vert b_i \vert^2
= \int_{\Ihat} (1-\xi^2)^k \vert u^{(k+2)}(\xi) \vert^2 d\xi
\le \Norm{u^{(k+2)}}_{0,\Ihat}^2.
\end{equation}
Using the orthogonality properties of Legendre polynomials~\cite[eq. (C.24)]{Schwab:1997},
the fact that~$s \le \p-2$,
and~\eqref{Vk-norm-identity},
we obtain
\[
\begin{split}
\Norm{u''-\up''}_{0,\Ihat}^2
& = \sum_{i=\p-1}^{+\infty} \frac{2}{2i+1} \vert b_i \vert^2
  = \sum_{i=\p-1}^{+\infty} \frac{2}{2i+1} \frac{(i+s)!}{(i-s)!} \vert b_i \vert^2 \frac{(i-s)!}{(i+s)!}\\
& \le \frac{(\p-s-1)!}{(\p+s-1)!} \sum_{i=s}^{+\infty} \frac{2}{2i+1} \frac{(i+s)!}{(i-s)!} \vert b_i \vert^2
\le \frac{(\p-s-1)!}{(\p+s-1)!} \Norm{u^{(s+2)}}_{0,\Ihat}^2.
\end{split}
\]
The two above equations yield the first bound in~\eqref{1D:hp-approx}.

Next, we introduce
\begin{equation} \label{writing:u:2}
\up'(\xi) = \int_{-1}^\xi \up''(t) dt + u'(-1).
\end{equation}
We have~$\up'(-1) = u'(-1)$.
Moreover, recalling that~$\up''$ and~$u''$ have the same average over~$\Ihat$, we also have
\[
\up'(1)
= \int_{-1}^1 \up''(t) dt + u'(-1)
= \int_{-1}^1 u''(t) dt + u'(-1)
= u'(1),
\]
which proves~\eqref{continuity:function&derivative} for the derivative of~$\up$ at the endpoints of~$\Ihat$.

At this point, we observe
\begin{equation} \label{difference-derivatives}
u'(\xi) - \up'(\xi)
= \int_{-1}^\xi (u''(t) - \up''(t)) dt
=  \sum_{i=\p-1}^{+\infty} b_i \int_{-1}^{\xi} L_i(t) dt
=: \sum_{i=\p-1}^{+\infty} b_i \phi_i(\xi).
\end{equation}
Recall the Legendre differential equation~\cite[eq. (C.2.3)]{Schwab:1997}
\[
((1-\xi^2)L_i'(\xi))' + i(i+1)L_i(\xi)=0
\qquad \forall i \in \Nbb_0.
\]
Integrating the above identity over~$(-1,\xi)$, $\xi \in (-1,1)$, yields
\begin{equation} \label{identity:phi-i}
\phi_i (\xi) = -\frac{1}{i(i+1)} (1-\xi^2) L_i'(\xi).
\end{equation}
Recall the orthogonality property of the derivatives of Legendre polynomials~\cite[eq. (3.39)]{Schwab:1997}:
\begin{equation} \label{orthogonality:L-prime}
\int_{-1}^1 (1-\xi^2) L_i'(\xi) L_j'(\xi) d\xi
= \frac{2\delta_{i,j}}{2i+1} \frac{(i+1)!}{(i-1)!}.
\end{equation}
Combining~\eqref{identity:phi-i} and~\eqref{orthogonality:L-prime},
for all~$i$, $j \in \Nbb_0$, we deduce
\begin{equation} \label{phi-orthogonality}
\begin{split}
\int_{-1}^1 (1-\xi^2)^{-1} \phi_i(\xi) \phi_j(\xi) d\xi
& = \frac{1}{i(i+1)j(j+1)} \int_{-1}^1 (1-\xi^2) L_i'(\xi) L_j'(\xi) d\xi \\
& = \frac{1}{i^2(i+1)^2} \frac{2\delta_{i,j}}{2i+1} \frac{(i+1)!}{(i-1)!}
  = \frac{2\delta_{i,j}}{i(i+1)(2i+1)}.
\end{split}
\end{equation}
Using~$s \le \p-2$,
we write
\[
\begin{split}
& \int_{-1}^1 \vert u'(\xi) -\up'(\xi) \vert^2 d\xi
  = \int_{-1}^1 \vert \!\!\!\sum_{i=\p-1}^{+\infty} b_i \phi_i(\xi) \vert^2 d\xi
 \le \int_{-1}^1 (1-\xi^2)^{-1} \vert \!\!\! \sum_{i=\p-1}^{+\infty} b_i \phi_i(\xi) \vert^2 d\xi\\
& \overset{\eqref{phi-orthogonality}}{\le}
\sum_{i=\p-1}^{+\infty} \frac{2\vert b_i \vert^2}{i(i+1)(2i+1)}
  = \sum_{i=\p-1}^{+\infty}
    \left( \frac{2\vert b_i \vert^2}{2i+1} \frac{(i+s)!}{(i-s)!} \right)  \frac{(i-s)!}{(i+s)!} \frac{1}{i(i+1)}\\
& \le \frac{(\p-s-1)!}{(\p+s-1)!} \frac{1}{(\p-1)\p}
     \left( \sum_{i=s}^{+\infty} \frac{2\vert b_i \vert^2}{2i+1} \frac{(i+s)!}{(i-s)!}  \right)
  \overset{\eqref{Vk-norm-identity}}{\le}
  \frac{(\p-s-1)!}{(\p+s-1)!} \frac{1}{(\p-1)\p} \Norm{u^{(s+2)}}_{0,\Ihat}^2,
\end{split}
\]
which is the second bound in~\eqref{1D:hp-approx}.

Finally, we introduce
\begin{equation} \label{writing:u:3}
\up(\xi) = \int_{-1}^\xi \up'(t) dt + u(-1).
\end{equation}
We observe that~$\up(-1) = u(-1)$.
Since~$L_1'(t)=1=L_0(t)$, standard manipulations imply
\[
\begin{split}
\up(1)-\up(-1)
& = \int_{-1}^1 \up'(t)dt
  \overset{\eqref{writing:u:2}}{=}
   2 u'(-1) + \int_{-1}^1 \left(\int_{-1}^t \up''(x)dx \right) dt \\
& \overset{\eqref{expanding:u-second},\eqref{difference-derivatives}}{=}
    2 u'(-1) + \int_{-1}^1 \sum_{i=0}^{\p-2} b_i \phi_i(t) dt
  = 2 u'(-1) + \sum_{i=0}^{\p-2} b_i \int_{-1}^1 \phi_i(t) dt \\
& \overset{\eqref{identity:phi-i}}{=} 2 u'(-1) - \sum_{i=0}^{\p-1} b_i \int_{-1}^1 \frac{1}{i(i+1)} (1-t^2) L_i'(t) L_1'(t) dt\\
& \overset{\eqref{orthogonality:L-prime}}{=} 2 u'(-1) - \sum_{i=0}^{+\infty} b_i \int_{-1}^1 \frac{1}{i(i+1)} (1-t^2) L_i'(t) L_1'(t) dt\\
& \overset{\eqref{identity:phi-i},\eqref{expanding:u-second}}{=} 2 u'(-1) + \int_{-1}^1 \left( \int_{-1}^t u''(x)dx  \right)dt
  = \int_{-1}^1 u'(t) dt = u(1)-u(-1).
\end{split}
\]
Using that~$\up(-1) = u(-1)$, we deduce~$\up(1)=u(1)$.

We are left with proving error estimates in the $L^2$ norm.
To this aim, observe that
\[
u(\xi)-\up(\xi)
= \int_{-1}^\xi (u'(t) - \up'(t)) dt.
\]
We arrive at
\begin{equation} \label{expansion:psi}
u(\xi)-\up(\xi)
\overset{\eqref{difference-derivatives}}{=}
 \sum_{i=\p-1}^{+\infty} b_i
  \int_{-1}^{\xi} \phi_i(x) \ dx \ dt
=: \sum_{i=\p-1}^{+\infty} b_i \psi_i(\xi).
\end{equation}
We prove certain orthogonality properties of the $\psi_i$ functions.
Recall the identity~\cite[eq. (C.2.5)]{Schwab:1997}:
\[
L_i(\xi)
= \frac{L'_{i+1}(\xi) - L'_{i-1}(\xi)}{2i+1} \qquad \forall i\ge 2,
\qquad \qquad
L_0(\xi)=L_1'(\xi).
\]
Integrating over~$(-1,t)$, $t\in(-1,1)$,
and using~$L_{i+1}(-1)=L_{i-1}(-1)$,
see~\cite[eq. (C.2.6)]{Schwab:1997},
we deduce
\[
\phi_i(t)= \frac{L_{i+1}(\xi) - L_{i-1}(\xi)}{2i+1}.
\]
Upon integrating the above identity over~$(-1,\xi)$, $\xi\in(-1,1)$, we arrive at
\[
\psi_i(\xi)
= \frac{\int_{-1}^\xi L_{i+1}(t)\ dt
                - \int_{-1}^\xi L_{i-1}(t) \ dt}{2i+1}
\overset{\eqref{difference-derivatives}}{=}
\frac{\phi_{i+1}(\xi) - \phi_{i-1}(t)}{2i+1}.
\]
With the notation~$\phi_{-1}(\xi)=0$,
we get
\[
\begin{split}
& \psi_i(\xi) \psi_j(\xi)\\
& = \frac{1}{(2i+1)(2j+1)}
\left(
\phi_{i+1}(\xi) \phi_{j+1}(\xi)
+ \phi_{i-1}(\xi) \phi_{j-1}(\xi)
- \phi_{i+1}(\xi) \phi_{j-1}(\xi)
- \phi_{i-1}(\xi) \phi_{j+1}(\xi)
\right).
\end{split}
\]
In~\eqref{phi-orthogonality}, we proved that the~$\phi_i$ functions are orthogonal with respect to the $(1-\xi^2)$-weighted $L^2$ inner product.
Therefore, testing the above identity by~$(1-\xi^2)^{-1}$
and integrating over~$(-1,1)$ we arrive at
\begin{equation} \label{psi-orthogonality:tridiagonal}
\int_{-1}^1 (1-\xi^2)^{-1} \psi_i(\xi) \psi_j(\xi) \ d\xi
= \frac{1}{(2i+1)(2j+1)}
\left( \aleph_{i+1,j+1}
+ \aleph_{i-1,j-1}
- \aleph_{i+1,j-1}
- \aleph_{i-1,j+1} \right),
\end{equation}
where, from~\eqref{phi-orthogonality},
\[
\aleph_{\ell, k}
=
\begin{cases}
0 & \text{if } \ell\le 0 \text{ or } k\le 0,\\
\frac{2\delta_{\ell,k}}{\ell(\ell+1)(2\ell+1)} & \text{if } \ell, \ k>0.
\end{cases}
\]
We provide explicit values for the expression in~\eqref{psi-orthogonality:tridiagonal}.
If~$j \in \{ i-2, i, i+2 \}$, then
\[
\int_{-1}^1 (1-\xi^2)^{-1} \psi_i(\xi) \psi_j(\xi) \ d\xi = 0.
\]
If~$j=i$, then
\[
\int_{-1}^1 (1-\xi^2)^{-1} \psi_i(\xi) \psi_i(\xi) \ d\xi
=
\beth^i_1 + \beth^i_2,
\]
where
\begin{equation} \label{beth-1-2}
\beth_1^i :=
\frac{1}{(2i+1)^2}  \frac{2}{(i+1)(i+2)(2i+3)}  \quad \forall i \in \Nbb;
\qquad
\beth_2^i:=
\begin{cases}
0                                           & \text{if } i=1, \\
\frac{1}{(2i+1)^2} \frac{2}{(i-1) i (2i-1)} & \text{if } i\ge 2.
\end{cases}
\end{equation}
If~$j=i+2$, then
\begin{equation} \label{beth-3}
\int_{-1}^1 (1-\xi^2)^{-1} \psi_i(\xi) \psi_{i+2}(\xi) \ d\xi
= - \frac{1}{(2i+1)(2i+5)} \cdot
  \frac{2}{(i+1)(i+2)(2i+3)}
=: -\beth^i_3.
\end{equation}
If~$j=i-2$, then
\begin{equation} \label{beth-4}
\int_{-1}^1 (1-\xi^2)^{-1} \psi_i(\xi) \psi_{i-2}(\xi) \ d\xi
= - \beth^i_4 :=
\begin{cases}
- \frac{1}{(2i+1)(2i-3)}\cdot
  \frac{2}{(i-1) i (2i-1)}  & \text{if } i\ge3,\\
0 & \text{if } i\le2.
\end{cases}
\end{equation}
In light of the above orthogonality properties,
we can estimate the $L^2$ approximation error as follows:
\[
\begin{split}
\Norm{u-\up}^2_{0,\Ihat}
& \overset{\eqref{expansion:psi}}{=}
    \Norm{\sum_{i=\p-1}^{{+\infty}} b_i \psi_i}^2_{0,\Ihat}
  \le \sum_{i,j=\p-1}^{+\infty}
    \SemiNorm{b_i \ b_j \int_{-1}^1 (1-\xi^2)^{-1} \psi_i(\xi) \psi_j(\xi) d\xi }\\
& \overset{\eqref{beth-1-2}, \eqref{beth-3}, \eqref{beth-4}}{\le}
    \sum_{i=\p-1}^{+\infty} \vert b_i \vert^2
            \left( \beth^i_1 + \beth^i_2 + \beth^i_3 + \beth^i_4  \right).
\end{split}
\]
We cope with the four terms on the right-hand side separately.
We begin with~$\beth^i_1$:
\[
\begin{split}
\sum_{i=\p-1}^{+\infty} \vert b_i \vert^2 \beth^i_1
& \overset{\eqref{beth-1-2}}{=}
    \sum_{i=\p-1}^{+\infty} \vert b_i \vert^2
        \frac{1}{(2i+1)^2} \cdot
        \frac{2}{(i+1)(i+2)(2i+3)}\\
& = \sum_{i=\p-1}^{+\infty}
    \left( \frac{2}{2i+1} \vert b_i \vert^2 \frac{(i+s)!}{(i-s)!}   \right)
    \left( \frac{(i-s)!}{(i+s)!} \frac{1}{(i+1)(i+2)(2i+1)(2i+3)}  \right)\\
&  \overset{\eqref{Vk-norm-identity}}{\le}
\frac{(\p-s-1)!}{(\p+s-1)!} \frac{1}{\p(\p+1)(2\p-1)(2\p+1)} \Norm{u^{(s+2)}}_{0,\Ihat}^2.
\end{split}
\]
Next, we focus on~$\beth^i_2$, $i\ge2$:
\[
\begin{split}
\sum_{i=\max(2,\p-1)}^{+\infty} \vert b_i \vert^2 \beth^i_2
& \overset{\eqref{beth-1-2}}{=}
\sum_{i=\max(2,\p-1)}^{+\infty} \vert b_i \vert^2
        \frac{1}{(2i+1)^2} \cdot
        \frac{2}{(i-1) i (2i-1)}\\
& = \sum_{i=\max(2,\p-1)}^{+\infty}
    \left( \frac{2}{2i+1} \vert b_i \vert^2 \frac{(i+s)!}{(i-s)!}   \right)
    \left( \frac{(i-s)!}{(i+s)!} \frac{1}{(i-1) i (2i-1)(2i+1)}  \right)\\
& \overset{\eqref{Vk-norm-identity}}{\le}
\frac{(\p-s-1)!}{(\p+s-1)!} \frac{1}{(\p-2)(\p-1)(2\p-3)(2\p-1)} \Norm{u^{(s+2)}}_{0,\Ihat}^2.
\end{split}
\]
As for the term~$\beth^i_3$, we proceed as follows:
\[
\begin{split}
\sum_{i=\p-1}^{+\infty} \vert b_i \vert^2 \beth^i_3
& \overset{\eqref{beth-3}}{=}
\sum_{i=\p-1}^{+\infty} \vert b_i \vert^2
        \frac{1}{(2i+1)(2i+5)} \cdot
        \frac{2}{(i+1)(i+2)(2i+3)}\\
& = \sum_{i=\p-1}^{+\infty}
    \left( \frac{2}{2i+1} \vert b_i \vert^2 \frac{(i+s)!}{(i-s)!}   \right)
    \left( \frac{(i-s)!}{(i+s)!} \frac{1}{(i+1)(i+2)(2i+3)(2i+5)}  \right)\\
& \overset{\eqref{Vk-norm-identity}}{\le}
\frac{(\p-s-1)!}{(\p+s-1)!} \frac{1}{\p(\p+2)(2\p+1)(2\p+3)} \Norm{u^{(s+2)}}_{0,\Ihat}^2.
\end{split}
\]
Eventually, we cope with the term~$\beth^i_4$, $i\ge3$:
\[
\begin{split}
\sum_{i=\max(3,\p-1)}^{+\infty} \vert b_i \vert^2 \beth^i_4
& \overset{\eqref{beth-4}}{=}
\sum_{i=\max(3,\p-1)}^{+\infty} \vert b_i \vert^2
        \frac{1}{(2i+1)(2i-3)} \cdot
        \frac{2}{(i-1) i (2i-1)}\\
& = \sum_{i=\p-1}^{+\infty}
    \left( \frac{2}{2i+1} \vert b_i \vert^2 \frac{(i+s)!}{(i-s)!}   \right)
    \left( \frac{(i-s)!}{(i+s)!} \frac{1}{(i-1) i (2i-3)(2i-1)} \right)\\
& \le \frac{(\p-s-1)!}{(\p+s-1)!} \frac{1}{(\p-2)(\p-1)(2\p-5)(2\p-3)} \Norm{u^{(s+2)}}_{0,\Ihat}^2.
\end{split}
\]
Collecting the five bounds above concludes the proof.

\section{Proof of Theorem~\ref{theorem:H2-2D/3D} (2D case)} \label{appendix:theorem:H2-2D}
The continuity properties~\eqref{continuity:2D} follow from the definition of the operator~$\Pcalp \cdot$.
Therefore, we only focus on the  error estimates.
We split the proof into several steps.
Recall that~$\Pcalpx$ and~$\Pcalpy$ are the projections
in Lemma~\ref{lemma:H2-1D}
along the~$x$ and~$y$ directions, respectively.

\paragraph*{Some identities.}
The following identities are valid:
\begin{equation} \label{essential-commutative-property:2D}
\partialy \Pcalpx u
= \Pcalpx (\partialy u),
\qquad
\partialx \Pcalpy u
= \Pcalpy (\partialx u).
\end{equation}
We only show the first one as the second can be proven analogously.
For all~$y \in (-1,1)$,
after writing~$\Pcalpx u (x,y)$ in integral form with respect to the~$x$ variable,
we have
\[
\begin{split}
& \Pcalpx u ( x,y )
 = \int_{-1}^x \left[ \left( \int_{-1}^t \partialx^2 \Pcalpx u (x,y) dx \right)  + \partialx u(-1,y) \right] dt + u(-1,y)\\
& \overset{\eqref{continuity:function&derivative}}{=}
    \int_{-1}^x \left[ \left( \int_{-1}^t \partialx^2 \Pcalpx u (x,y) dx \right)  + \partialx u(-1,y) \right] dt + u(-1,y)\\
& \overset{\eqref{writing:u:1},\eqref{equation:Legendre-properties}}{=}
    \int_{-1}^x \left[ \left( \int_{-1}^t \partialx^2
    \left(\sum_{i=0}^{\p-2} \frac{2i+1}{2} \int_{-1}^1 u(s,y) L_i(s) ds \right) L_i(x)
    dx \right)  + \partialx u(-1,y) \right] dt + u(-1,y).
\end{split}
\]
Taking the partial derivative with respect to~$y$ on both sides yields
\[
\begin{split}
& \partialy \Pcalpx u ( x,y )\\
& =  \int_{-1}^x \left[ \left( \int_{-1}^t \partialx^2
    \left(\sum_{i=0}^{\p-2} \frac{2i+1}{2} \int_{-1}^1 \partialy u(s,y) L_i(s) ds \right) L_i(x)
    dx \right)  + \partialx \partialy u(-1,y) \right] dt + \partialy u(-1,y)\\
& \overset{\eqref{writing:u:1},\eqref{writing:u:2},\eqref{writing:u:3}}{=}  \Pcalpx \partialy u ( x,y ).
\end{split}
\]

\paragraph*{$L^2$ estimates.}
Using the definition of~$\Pcalp$,
the triangle inequality,
the one dimensional approximation properties~\eqref{1D:hp-approx:cor},
the third stability property in~\eqref{stability:Pcalp},
and the identities~\eqref{essential-commutative-property:2D}
we write
\[
\begin{split}
\Norm{u-\Pcalp u}_{0,\Qhat}
& \le \Norm{u-\Pcalpx u}_{0,\Qhat} + \Norm{\Pcalpx(u-\Pcalpy u)}_{0,\Qhat}\\
& \lesssim \p^{-s-2} \Norm{\partialx^{s+2} u}_{0,\Qhat}
            + \Norm{u-\Pcalpy u}_{0,\Qhat} + \p^{-2} \Norm{\partialx^2 (u-\Pcalpy u)}_{0,\Qhat}\\
& = \p^{-s-2} \Norm{\partialx^{s+2} u}_{0,\Qhat}
            + \Norm{u-\Pcalpy u}_{0,\Qhat} + \p^{-2} \Norm{\partialx^2 u-\Pcalpy \partialx^2 u}_{0,\Qhat}.
\end{split}
\]
The assertion follows using again the one dimensional approximation properties~\eqref{1D:hp-approx:cor}.

\paragraph*{$H^1$ estimates.}
First, we cope with the bound on the derivative with respect to~$x$.
Adding and subtracting~$\Pcalpx u$, and using the triangle inequality yield
\begin{equation} \label{estimate:T1T2}
\begin{split}
\Norm{\partialx (u-\Pcalp u)}_{0,\Qhat}
& \le \Norm{\partialx (u-\Pcalpx u)}_{0,\Qhat}
      + \Norm{\partialx \Pcalpx(u- \Pcalp u)}_{0,\Qhat}
  =: T_1 + T_2.
\end{split}
\end{equation}
As for the term~$T_1$, we use the one dimensional approximation properties~\eqref{1D:hp-approx:cor} and get
\begin{equation} \label{estimate:T1}
T_1 \lesssim \p^{-s-1} \Norm{\partialx^{s+2} u}_{0,\Qhat}.
\end{equation}
As for the term~$T_2$, we use the second stability estimate in~\eqref{stability:Pcalp} and get
\[
T_2
\lesssim \Norm{\partialx (u-\Pcalpy u)}_{0,\Qhat}
             + \p^{-1} \Norm{\partialx^2 (u-\Pcalpy u)}_{0,\Qhat}.
\]
Thanks to the identities~\eqref{essential-commutative-property:2D}
and the one dimensional approximation properties~\eqref{1D:hp-approx:cor},
we can estimate the term~$T_2$ from above as follows:
\begin{equation} \label{estimate:T2}
T_2
 \lesssim \Norm{\partialx u - \Pcalpy \partialx u}_{0,\Qhat}
         + \p^{-1} \Norm{\partialx^2 u - \Pcalpy \partialx^2 u}_{0,\Qhat}
  \lesssim \p^{-s-1}
            \left( \Norm{\partialx \partialy^{s+1} u}_{0,\Qhat}
            + \Norm{\partialx^2 \partialy^{s} u}_{0,\Qhat} \right).
\end{equation}
Collecting the estimates~\eqref{estimate:T1} and~\eqref{estimate:T2} in~\eqref{estimate:T1T2}, we arrive at
\[
\Norm{\partialx (u-\Pcalp u)}_{0,\Qhat}
\lesssim \p^{-s-1} \left(
        \Norm{\partialx^{s+2}u}_{0,\Qhat}
        + \Norm{\partialx \partialy^{s+1}u}_{0,\Qhat}
        + \Norm{\partialx^2 \partialy^{s}u}_{0,\Qhat} \right).
\]
With similar arguments for the $y$ derivative term,
we deduce~\eqref{2D:hp-approx:H1}.

\paragraph*{$H^2$ estimates.}
We begin by showing an upper bound on the second derivative with respect to~$x$.
Using the triangle inequality,
the one dimensional approximation properties~\eqref{1D:hp-approx:cor},
the stability properties~\eqref{stability:Pcalp},
and the identities~\eqref{essential-commutative-property:2D},
we obtain
\[
\begin{split}
\Norm{\partialx^2 (u-\Pcalp u)}_{0,\Qhat}
& \le \Norm{\partialx^2 (u-\Pcalpx u)}_{0,\Qhat}
     + \Norm{\partialx^2 \Pcalpx (u-\Pcalpy u)}_{0,\Qhat}\\
& \lesssim \p^{-s} \Norm{\partialx^{s+2} u}_{0,\Qhat}
            + \Norm{\partialx^2 (u-\Pcalpy u)}_{0,\Qhat}
    = \p^{-s} \Norm{\partialx^{s+2} u}_{0,\Qhat}
            + \Norm{\partialx^2 u-\Pcalpy \partialx^2 u}_{0,\Qhat}\\
&  \lesssim \p^{-s} \left( \Norm{\partialx^{s+2} u}_{0,\Qhat}
        + \Norm{\partialx^2 \partialy^{s} u}_{0,\Qhat} \right).
\end{split}
\]
Analogously, we can prove
\[
\Norm{\partialy^2 (u-\Pcalp u)}_{0,\Qhat}
\lesssim \p^{-s} \left( \Norm{\partialy^{s+2} u}_{0,\Qhat}
        + \Norm{\partialx^{s} \partialy^2 u}_{0,\Qhat} \right).
\]
Eventually, we cope with the mixed derivative term.
Using the triangle inequality and the identities~\eqref{essential-commutative-property:2D},
we get
\begin{equation} \label{estimate:S1S2}
\begin{split}
\Norm{\partialx \partialy (u-\Pcalp u)}_{0,\Qhat}
& \le \Norm{\partialx \partialy (u-\Pcalpx u)}_{0,\Qhat}
      + \Norm{\partialx \partialy \Pcalpx (u-\Pcalpy u)}_{0,\Qhat}\\
& = \Norm{\partialx (\partialy u-\Pcalpx \partialy u)}_{0,\Qhat}
      + \Norm{\partialx  \Pcalpx \partialy (u-\Pcalpy u)}_{0,\Qhat}
      =: S_1 + S_2.
\end{split}
\end{equation}
We estimate the two terms~$S_1$ and~$S_2$ separately.
Using the one dimensional approximation properties~\eqref{1D:hp-approx:cor},
we can write
\begin{equation} \label{estimate:S1}
\begin{split}
S_1
& \lesssim \p^{-s} \Norm{\partialx^{s+1} \partialy u}_{0,\Qhat}.
\end{split}
\end{equation}
On the other hand,
using the one dimensional approximation properties~\eqref{1D:hp-approx:cor}
and the stability properties~\eqref{stability:Pcalp}, we get
\begin{equation} \label{estimate:S2}
\begin{split}
S_2
& \lesssim \Norm{\partialx\partialy (u-\Pcalpy y)}_{0,\Qhat}
            + \p^{-1} \Norm{\partialx^2 \partialy (u-\Pcalpy u)}_{0,\Qhat}\\
& = \Norm{\partialy (\partialx u-\Pcalpy \partialx y)}_{0,\Qhat}
    + \p^{-1} \Norm{\partialy (\partialx^2 u-\Pcalpy \partialx^2 u)}_{0,\Qhat}\\
& \lesssim \p^{-s}
    \left( \Norm{\partialx\partialy^{s+1}u}_{0,\Qhat}
    + \Norm{\partialx^2 \partialy^{s} u}_{0,\Qhat} \right).
\end{split}
\end{equation}
Collecting the estimates~\eqref{estimate:S1} and~\eqref{estimate:S2} in~\eqref{estimate:S1S2},
we arrive at
\[
\Norm{\partialx \partialy (u-\Pcalp u)}_{0,\Qhat}
\lesssim
\p^{-s}
    \left( \Norm{\partialx^{s+1} \partialy u}_{0,\Qhat}
    + \Norm{\partialx\partialy^{s+1}u}_{0,\Qhat}
    + \Norm{\partialx^2 \partialy^{s} u}_{0,\Qhat} \right).
\]
Combining the estimates on all second derivatives terms,
we obtain~\eqref{2D:hp-approx:H2}.

\section{Proof of Theorem~\ref{theorem:H2-2D/3D} (3D case)} \label{appendix:theorem:H2-3D}

The continuity properties follow exactly as in the two dimensional case.
Thus, we only show the details for the approximation properties.
We split the proof in several steps.
Recall that~$\Pcalpx$, $\Pcalpy$, and~$\Pcalpz$ are the projections
in Lemma~\ref{lemma:H2-1D}
along the~$x$, $y$, and~$z$ directions, respectively.

\paragraph*{Some identities.}
Analogous to their two dimensional counterparts in~\eqref{essential-commutative-property:2D},
we have the following identities:
\begin{equation} \label{essential-commutative-property:3D}
\begin{split}
& \partialy \Pcalpx u
= \Pcalpx (\partialy u),
\qquad
\partialx \Pcalpy u
= \Pcalpy (\partialx u),
\qquad
\partialy \Pcalpz u
= \Pcalpz (\partialy u),\\
& \partialz \Pcalpy u
= \Pcalpy (\partialz u),
\qquad
\partialx \Pcalpz u
= \Pcalpz (\partialx u),
\qquad
\partialz \Pcalpx u
= \Pcalpx (\partialz u).
\end{split}
\end{equation}

\paragraph*{$L^2$ estimates.}
The triangle inequality and the one dimensional approximation properties~\eqref{1D:hp-approx:cor} imply
\[
\begin{split}
\Norm{u-\Pcalp u}_{0,\Qhat}
&   \!\le\! \Norm{u-\Pcalpx u}_{0,\Qhat}
        + \Norm{\Pcalpx(u-\Pcalpy\Pcalpz u)}_{0,\Qhat}
    \!\lesssim\! \p^{-s-2}\Norm{\partialx^{s+2}u}_{0,\Qhat}
        \!\!\!+ \Norm{\Pcalpx(u-\Pcalpy\Pcalpz u)}_{0,\Qhat}.
\end{split}
\]
We focus on the second term.
To this aim, we use the stability properties~\eqref{stability:Pcalp},
the triangle inequality,
the one dimensional approximation properties~\eqref{1D:hp-approx:cor},
and the identities~\eqref{essential-commutative-property:3D},
and deduce
\[
\begin{split}
& \Norm{\Pcalpx(u-\Pcalpy\Pcalpz u)}_{0,\Qhat}
 \le \Norm{u - \Pcalpy \Pcalpz u}_{0,\Qhat}
      + \p^{-2} \Norm{\partialx^2 u - \Pcalpy \Pcalpz \partialx^2 u}_{0,\Qhat}\\
& \lesssim \Norm{u - \Pcalpy u}_{0,\Qhat}
          + \Norm{\Pcalpy(u - \Pcalpz u)}_{0,\Qhat}
          + \p^{-2} \Norm{\partialx^2 u - \Pcalpy \partialx^2 u}_{0,\Qhat}
          + \p^{-2} \Norm{\Pcalpy(\partialx^2 u - \Pcalpz \partialx^2 u)}_{0,\Qhat} \\
& \lesssim \p^{-s -2}  \Norm{\partialy^{s+2}u}_{0,\Qhat}
+ \Norm{u-\Pcalpz u}_{0,\Qhat}
+ \p^{-2} \Norm{\partialy^2 u - \Pcalpz \partialy^2 u}_{0,\Qhat}
+ \p^{-s-4} \Norm{\partialx^2\partialy^2\partialz^{s-2}u}_{0,\Qhat}\\
& \quad +\p^{-2} \Norm{\partialx^2u-\Pcalpz \partialx^2 u}_{0,\Qhat}
     + \p^{-4} \Norm{\partialx^2 \partialy^2 u - \Pcalpz \partialx^2 \partialy^2 u}_{0,\Qhat}.
\end{split}
\]
The $L^2$ estimates eventually follow from the one dimensional approximation properties~\eqref{1D:hp-approx:cor}.

\paragraph*{$H^1$ estimates.}
We show the details for the $x$ derivative, as the other two cases can be dealt with analogously.
The triangle inequality and the one dimensional approximation properties~\eqref{1D:hp-approx:cor} imply
\[
\begin{split}
\Norm{\partialx (u -\Pcalp u)}_{0,\Qhat}
& \lesssim \Norm{\partialx (u -\Pcalpx u)}_{0,\Qhat}
        + \Norm{\partialx \Pcalpx (u -\Pcalpy\Pcalpz u)}_{0,\Qhat}\\
& \lesssim \p^{-s-1} \Norm{\partialx^{s+2}u}_{0,\Qhat}
        + \Norm{\partialx \Pcalpx (u -\Pcalpy\Pcalpz u)}_{0,\Qhat}.
\end{split}
\]
We focus on the second term on the right-hand side.
The stability properties~\eqref{stability:Pcalp}
and the identities~\eqref{essential-commutative-property:3D} entail
\[
\Norm{\partialx \Pcalpx (u -\Pcalpy\Pcalpz u)}_{0,\Qhat}
\lesssim \Norm{\partialx u - \Pcalpy\Pcalpz \partialx u}_{0,\Qhat}
        + \p^{-2} \Norm{\partialx^2 u - \Pcalpy\Pcalpz \partialx^2 u}_{0,\Qhat}
=: T_1+T_2.
\]
As for the term~$T_1$,
the triangle inequality,
the identities~\eqref{essential-commutative-property:3D},
and the one dimensional approximation properties~\eqref{1D:hp-approx:cor} imply
\[
\begin{split}
T_1
& \!\lesssim\! \Norm{\partialx u \!-\! \Pcalpy \partialx u}_{0,\Qhat}
       + \Norm{\Pcalpy(\partialx u \!-\! \Pcalpz \partialx u)}_{0,\Qhat}
\!\lesssim\! \p^{-s-1} \Norm{\partialx \partialy^{s+1} u }_{0,\Qhat}
       + \Norm{\Pcalpy(\partialx u \!-\! \Pcalpz \partialx u)}_{0,\Qhat}.
\end{split}
\]
The second term on the right-hand side can be estimated
using the stability properties~\eqref{stability:Pcalp},
the identities~\eqref{essential-commutative-property:3D},
and the one dimensional approximation properties~\eqref{1D:hp-approx:cor}:
\[
\begin{split}
& \Norm{\Pcalpy(\partialx u \!-\! \Pcalpz \partialx u)}_{0,\Qhat}
\lesssim \Norm{\partialx u \!-\! \Pcalpz \partialx u}_{0,\Qhat}
          + \p^{-2} \Norm{\partialx\partialy^2 u \!-\! \Pcalpz \partialx \partialy^2 u)}_{0,\Qhat}\\
& \lesssim \p^{-s-1}
        \left( \Norm{\partialx \partialz^{s+1}u}_{0,\Qhat}
    +\Norm{\partialx\partialy^2 \partialx^{s-1}u}_{0,\Qhat} \right).
\end{split}
\]
With similar arguments based on substituting $\partialx u$ by~$\partialx^2 u$, we find an upper bound for~$T_2$:
\[
T_2
\lesssim \p^{-s-1}
        \left( \Norm{\partialx^2 \partialy^{s}u}_{0,\Qhat}
        + \Norm{\partialx^2 \partialz^{s}u}_{0,\Qhat}
        + \Norm{\partialx^2\partialy^2 \partialz^{s-2}u}_{0,\Qhat} \right).
\]
Collecting the above estimates, we arrive at
\[
\begin{split}
& \Norm{\partialx (u -\Pcalp u)}_{0,\Qhat}
\lesssim \p^{-s-1} 
\Big(
 \Norm{\partialx^{s+2} u}_{0,\Qhat}
 + \Norm{\partialx \partialy^{s+1} u}_{0,\Qhat}
 + \Norm{\partialx \partialz^{s+1} u}_{0,\Qhat} \\
& \quad
 + \Norm{\partialx \partialy^2 \partialz^{s-1} u}_{0,\Qhat}
 + \Norm{\partialx^2 \partialy^s u}_{0,\Qhat}
 + \Norm{\partialx^2 \partialz^{s} u}_{0,\Qhat}
 + \Norm{\partialx^2 \partialy^2 \partialz^{s-2} u}_{0,\Qhat}
 \Big).
\end{split}
\]
By similar arguments on the $y$ and~$x$ partial derivatives, we deduce~\eqref{3D:hp-approx:H1}.

\paragraph*{$H^2$ estimates.}
First, we show the details for the second $x$ derivative, since the second $y$ and~$z$ derivatives can be dealt with analogously.
The triangle inequality,
the one dimensional approximation properties~\eqref{1D:hp-approx:cor},
and the identities~\eqref{essential-commutative-property:3D} imply
\[
\begin{split}
\Norm{\partialx^2 (u - \Pcalp u)}_{0,\Qhat}
& \le \Norm{\partialx^2 (u - \Pcalpx u)}_{0,\Qhat}
      + \Norm{\partialx^2 \Pcalpx (u - \Pcalpy\Pcalpz u)}_{0,\Qhat}\\
& \lesssim \p^{-s} \Norm{\partialx^{s+2}u}_{0,\Qhat}
            + \Norm{\partialx^2 u - \Pcalpy\Pcalpz \partialx^2 u}_{0,\Qhat}.
\end{split}
\]
We are left with estimating the second term on the right-hand side.
Applying a further triangle inequality,
the one dimensional approximation properties~\eqref{1D:hp-approx:cor},
and the stability properties~\eqref{stability:Pcalp}, we arrive at
\[
\begin{split}
\Norm{\partialx^2 u - \Pcalpy\Pcalpz \partialx^2 u}_{0,\Qhat}
& \le   \Norm{\partialx^2 u - \Pcalpy \partialx^2 u}_{0,\Qhat}
        + \Norm{\Pcalpy(\partialx^2 u - \Pcalpz \partialx^2 u)}_{0,\Qhat}\\
& \lesssim \p^{-s} \Norm{\partialx^2 \partialy^{s}u}_{0,\Qhat}
            + \Norm{\partialx^2 u - \Pcalpz \partialx^2 u}_{0,\Qhat}
            + \p^{-2} \Norm{\partialx^2 \partialy^2 u - \Pcalpz\partialx^2\partialy^2 u}_{0,\Qhat}\\
& \lesssim \p^{-s} \left(
            \Norm{\partialx^2 \partialy^{s}u}_{0,\Qhat}
            + \Norm{\partialx^2 \partialz^{s}u}_{0,\Qhat}
            + \Norm{\partialx^2 \partialy^2 \partialz^{s-2}u}_{0,\Qhat}
            \right).
\end{split}
\]
Collecting the two above estimate above gives
\begin{equation} \label{second-derivatives:estimate:3D}
\Norm{\partialx^2 (u - \Pcalp u)}_{0,\Qhat}
 \lesssim \p^{-s} \left(
            \Norm{\partialx^{s+2}u}_{0,\Qhat}
            + \Norm{\partialx^2 \partialy^{s}u}_{0,\Qhat}
            + \Norm{\partialx^2 \partialz^{s}u}_{0,\Qhat}
            + \Norm{\partialx^2 \partialy^2 \partialz^{s-2}u}_{0,\Qhat}
            \right).
\end{equation}

Next, we focus on the mixed $xy$ derivative and observe that the~$yz$ and~$xz$ counterparts are dealt with analogously.
Using the triangle inequality,
the one dimensional approximation properties~\eqref{1D:hp-approx:cor},
the identities~\eqref{essential-commutative-property:3D},
and the stability properties~\eqref{stability:Pcalp} leads us to
\[
\begin{split}
& \Norm{\partialx\partialy u - \Pcalp \partialx\partialy u}_{0,\Qhat}
 \le  \Norm{\partialx\partialy u - \Pcalpx \partialx\partialy u}_{0,\Qhat}
        + \Norm{\Pcalpx (\partialx\partialy u - \Pcalpy\Pcalpz \partialx\partialy u)}_{0,\Qhat}\\
& \lesssim \p^{-s} \Norm{\partialx^{s+1}\partialy u}_{0,\Qhat}
    + \Norm{\partialy \partialx u - \partialy \Pcalpy\Pcalpz \partialx u}_{0,\Qhat}
    + \p^{-1} \Norm{\partialy \partialx^2 u - \partialy \Pcalpy\Pcalpz \partialx^2 u}_{0,\Qhat}\\
& = \p^{-s} \Norm{\partialx^{s+1}\partialy u}_{0,\Qhat} + T_3+T_4.
\end{split}
\]
We estimate the terms~$A$ and~$B$ separately.
First, we focus on~$A$.
Using the triangle inequality,
the one dimensional approximation properties~\eqref{1D:hp-approx:cor},
the identities~\eqref{essential-commutative-property:3D},
and the stability properties~\eqref{stability:Pcalp},
entail
\[
\begin{split}
T_3
& \le \Norm{\partialy \partialx u - \partialy \Pcalpy \partialx u}_{0,\Qhat}
      + \Norm{\partialy \Pcalpy \partialx u - \partialy \Pcalpy \partialx u}_{0,\Qhat}\\
& \lesssim \p^{-s} \Norm{\partialx \partialy^{s+1}u}_{0,\Qhat}
        + \Norm{\partialy\partialx u - \Pcalpz \partialy\partialx u}_{0,\Qhat}
        + \p^{-1} \Norm{\partialy^2\partialx u - \Pcalpz \partialy^2\partialx u}_{0,\Qhat}\\
& \lesssim \p^{-s}
    \left(
        \Norm{\partialx \partialy^{s+1}u}_{0,\Qhat}
        + \Norm{\partialx \partialy \partialz^{s}u}_{0,\Qhat}
        + \Norm{\partialx \partialy^2 \partialz^{s-1}u}_{0,\Qhat}
    \right).
\end{split}
\]
Next, we focus on the term~$B$.
Using the triangle inequality,
the one dimensional approximation properties~\eqref{1D:hp-approx:cor},
the identities~\eqref{essential-commutative-property:3D},
and the stability properties~\eqref{stability:Pcalp},
leads to
\[
\begin{split}
T_4
& \le \p^{-1} \Norm{\partialy \partialx^2 u - \partialy \Pcalpy \partialx^2 u}_{0,\Qhat}
      + \p^{-1} \Norm{\partialy \Pcalpy \partialx^2 u - \partialy \Pcalpy \Pcalpz \partialx^2 u}_{0,\Qhat}\\
& \lesssim \p^{-s} \Norm{\partialx^2\partialy^{s}}_{0,\Qhat}
        + \p^{-1} \Norm{\partialy \partialx^2 u - \Pcalpz \partialy \partialx^2 u}_{0,\Qhat}
        + \p^{-2} \Norm{\partialx^2 \partialy^2 u - \Pcalpz \partialx^2 \partialy^2 u}_{0,\Qhat}\\
& \lesssim \p^{-s}
        \left(
            \Norm{\partialx^2\partialy^{s}}_{0,\Qhat}
            + \Norm{\partialx^2\partialy \partialz^{s}}_{0,\Qhat}
            + \Norm{\partialx^2\partialy^2 \partialz^{s-2}}_{0,\Qhat}
        \right).
\end{split}
\]
Collecting the above estimates gives
\begin{equation} \label{mixed-derivatives:estimate:3D}
\begin{split}
\Norm{\partialx\partialy u - \Pcalp \partialx\partialy u}_{0,\Qhat}
\lesssim \p^{-s}
& \left( \Norm{\partialx^{s+1}\partialy u}_{0,\Qhat}
        + \Norm{\partialx \partialy^{s+1} u}_{0,\Qhat}
        + \Norm{\partialx \partialy \partialz^{s} u}_{0,\Qhat}
        \Norm{\partialx \partialy^2 \partialz^{s-1} u}_{0,\Qhat}
            \right.  \\
& \left. + \Norm{\partialx^2 \partialy^{s} u}_{0,\Qhat}
        + \Norm{\partialx^2 \partialy \partialz^{s-1} u}_{0,\Qhat}
        + \Norm{\partialx^2 \partialy^2 \partialz^{s-2} u}_{0,\Qhat}
        \right).
\end{split}
\end{equation}
Bound~\eqref{3D:hp-approx:H2}
follows combining the estimates on the second derivative term~\eqref{second-derivatives:estimate:3D},
the mixed derivative term~\eqref{mixed-derivatives:estimate:3D},
and the corresponding estimates for similar derivative terms.

\end{document}